\pdfoutput=1
\newif\ifpreprint%
\preprinttrue   

\ifpreprint

\documentclass[11pt]{article}

\usepackage{abstract} 

\usepackage{titlesec}
\def\sectionfont{\sffamily\Large\bfseries\boldmath}
\def\subsectionfont{\sffamily\large\bfseries\boldmath}
\def\paragraphfont{\sffamily\normalsize\bfseries\boldmath}
\titleformat*{\section}{\sectionfont}
\titleformat*{\subsection}{\subsectionfont}
\titleformat*{\subsubsection}{\paragraphfont}
\titleformat*{\paragraph}{\paragraphfont}
\titleformat*{\subparagraph}{\paragraphfont}
\usepackage[small,labelfont={bf,sf}]{caption}  
\usepackage{fullpage,graphicx,psfrag,url}
\usepackage{multirow}
\usepackage{enumitem}
\setlist{nolistsep}

\usepackage{booktabs}           
\usepackage{adjustbox} 

\newcounter{algorithmctr}[section]
\renewcommand{\thealgorithmctr}{\thesection.\arabic{algorithmctr}}

\newtheorem{theorem}{Theorem}
\newtheorem{lemma}{Lemma}
\newtheorem{assumption}{Assumption}
\newtheorem{definition}{Definition}

\makeatletter
\DeclareRobustCommand{\qed}{%
	\ifmmode 
	\else \leavevmode\unskip\penalty9999 \hbox{}\nobreak\hfill
	\fi
	\quad\hbox{\qedsymbol}}
\newcommand{\openbox}{\leavevmode
	\hbox to.77778em{%
		\hfil\vrule
		\vbox to.675em{\hrule width.6em\vfil\hrule}%
		\vrule\hfil}}
\newcommand{\qedsymbol}{\openbox}
\newenvironment{proof}[1][\proofname]{\par
	\normalfont
	\topsep6\p@\@plus6\p@ \trivlist
	\item[\hskip\labelsep\itshape
	#1.]\ignorespaces
}{%
	\qed\endtrivlist
}
\newcommand{\proofname}{Proof}
\makeatother

\usepackage[noend]{algpseudocode}
\usepackage[english]{babel}
\usepackage{algorithm}

\usepackage{mathtools,amsfonts}

\usepackage[xindy, acronyms]{glossaries}   

\ifpreprint
\title{\bfseries\sffamily Burer-Monteiro ADMM for Large-scale SDPs}
\author{Yuwen Chen and Paul Goulart}
\fi

\begin{document}

\ifpreprint
\maketitle
	\begin{abstract}
		We propose a bilinear decomposition for the Burer-Monteiro method and combine it with the standard Alternating Direction Method of Multipliers algorithm for semidefinite programming. Bilinear decomposition reduces the degree of the augmented Lagrangian from four to two, which makes each of the subproblems a quadratic programming and hence computationally efficient. Our approach is able to solve a class of large-scale SDPs with diagonal constraints. We prove that our ADMM algorithm converges globally to a first-order stationary point, and show by exploiting the negative curvature that the algorithm converges to a point within $O(1-1/r)$ of the optimal objective value. Additionally, the proximal variant of the algorithm can solve block-diagonally constrained SDPs with global convergence to a first-order stationary point. Numerical results show that both our ADMM algorithm and the proximal variant outperform the state-of-art Riemannian manifold algorithms and can reach the global optimum empirically.
	\end{abstract}
\fi

\section{Introduction}
\subsection{Related Work}
SDPs \cite{Vandenberghe96} are widely used in various fields such as control engineering \cite{Boyd94}, signal processing \cite{Huang10}, combinatorial optimization \cite{Boyd97,Benson00} and finance \cite{Gepp20}. While interior-point methods \cite{Helmberg96} can be used to solve small to medium size SDPs in polynomial time, they do not scale well with dimension of the problem. At present, several methods have been proposed to tackle the scalability of SDPs \cite{Majumdar20}. One is to exploit the chordal decomposition of a SDP \cite{Fukuda00,Vandenberghe15}, which typically reduce a single cone constraint to a collection of constraints on lower dimensional cones. This technique has been successfully implemented for interior-point methods \cite{Andersen10} and ADMM-based methods \cite{Zheng20,Garstka21}. Approaches based on diagonal dominance (dd) and scaled diagonal dominance (sdd) have been studied for solving SDP approximately and shown to be effective in the SDP relaxation of the sum-of-square problems \cite{Ahmadi19}. Another direction is to exploit low-rank information of a SDP \cite{Lemon16, Burer03, Freund15, Yurtsever21}, a popular variation of which is called the Burer-Monteiro method \cite{Burer03}, which substitutes $X \succeq 0$ with a low-rank factorization $X = V V^\top$ where $V \in \mathbb{R}^{n \times r}$ with $r \ll n$, saving both computational time and data storage. Frank-Wolfe algorithms \cite{Jaggi13}, which are projection-free and circumvent a full eigenvalue decomposition, have also been proposed to solve large-scale SDPs and exploit low-rank information via in-face exploitation \cite{Freund15}. Meanwhile, some work has been done on the storage issue of SDPs. The sketching method, which has been well-studied in numerical linear algebra \cite{Woodruff14}, is applied to SDPs \cite{Yurtsever21,Duchi20}. Based on approximate complementary slackness, \cite{Ding20} proposes a method to iteratively estimate a low-rank eigenspace and then optimize the SDP with reduced dimension. In addition, SDPs can be reformulated as a equivalent problem that can be solved efficiently, e.g, formulating diagonally constrained SDPs as a nonconvex QP \eqref{sphere-form} and solving it by block-coordinate descent methods~\cite{Wang17,Erdogdu21}. 

The Burer-Monteiro (BM) method is a popular low-rank method for SDPs and works well in practice \cite{Burer03}. For the standard primal SDP problem
\begin{align}
	\begin{aligned}
		\min \{\langle C, X \rangle: \langle A_i, X \rangle = b_i, i = 1, \dots, m, X \succeq 0 \},
	\end{aligned} \label{primal-sdp}
\end{align}
the classical BM approach replaces the positive semidefinite constraint $X \succeq 0 $ with the factorization $X = VV^\top$ where $V \in \mathbb{R}^{n \times r}$, so that the problem becomes
\begin{align}
	\min \{\langle C, VV^\top \rangle: \langle A_i, VV^\top \rangle = b_i, i = 1, \dots,m \}. \label{BM-approach}
\end{align}
for very large scale problems with low rank solutions. Compared with \eqref{primal-sdp}, \eqref{BM-approach} requires storage of $O(nr)$ of $V$ instead of $O(n^2)$ of $X$ and eliminates the semidefinite cone constraint that can be beneficial for computation. The problem \eqref{BM-approach} is nonconvex due to the existence of quadratic equality constraints and global optimality is hard to guarantee, although a global optimum is usually obtained using the BFGS algorithm \cite{Burer03}. Generally speaking, the rank $r$ is selected according to the rank-inequality constraint in SDPs \cite{Alizadeh97, Pataki98, Lemon16}. Such a choice ensures that any local minimum is a global optimum for almost all cost matrices $C$ in \eqref{SDP-form} \cite{Boumal16} and is tight as shown in \cite{Waldspurger20}.

The alternating direction method of multipliers (ADMM), which is closely related to the alternating direction method and the augmented Lagrangian method, has become very popular in recent years for use in the large-scale optimization \cite{Boyd11}. Convergence proofs of ADMM in convex optimization typically rely on the design of a Lyapunov function \cite{Boyd11},\cite{Jiang19} which is monotonically decreasing, or the interpretation of it as a variant of the Douglas Rachford (DR) algorithm \cite{Eckstein92,Banjac19,Garstka21} whose convergence is proved based on monotone operator theory \cite{Bauschke11}. ADMM often works extremely well even for nonconvex problems \cite{Lai14,Kovnatsky16}, such as nonnegative matrix factorization \cite{Lu17}, optimal power flow \cite{You14} and image reconstruction \cite{Barber21} but convergence analysis is much less mature. In recent years, some results have been developed for the convergence of nonconvex ADMM to first-order stationary points in limited situations \cite{Hong16, Wang19, Jiang19, Bolte14}.

\subsection{Our Contribution}
In this paper, our contribution are:
\begin{enumerate}
	\item We propose an ADMM algorithm with a novel bilinear decomposition to the Burer-Monteiro approach for diagonally constrained SDPs in Section~\ref{section-BM-ADMM}.  We prove convergence of this algorithm in \textit{value} in Section~\ref{section-1st-order-convergence-objective} and in \textit{iterates} in Section~\ref{section-1st-convergence-sequence}.
	\item In Section~\ref{section-negative-curvature}, we show that exploiting negative curvature guarantees convergence to a solution of $O(1-1/r)$ optimality.
	\item In Section~\ref{section-proximal-variant}, we introduce a proximal variant of the algorithm that can also solve a broader class of diagonally constrained SDPs with global convergence to a first-order stationary point.
	\item Finally, numerical results show both algorithm and its proximal variant perform better than the state-of-the-art Riemannian manifold algorithms and scales well w.r.t. the dimension of SDPs in Section~\ref{section-experiments}.
\end{enumerate}
A preliminary version of the results of Section~\ref{section-1st-order-convergence-objective} appeared in~\cite{Chen22}.

\subsection{Notation and Preliminaries}
The Frobenius norm is $\|A\|_F = \sqrt{\langle A,A \rangle} = \sqrt{\text{Tr}(A^\top A)}$, where $A_{i,j}$ is the entry at the $i-$th row and $j-$th column of matrix $A$. The induced operator $p$-norm of matrix $A \in \mathbb{R}^{n \times n}$ is $\|A\|_p = \sup_{\|x\| \ne 0} \frac{\|Ax\|_p}{\|x\|_p}$, where $\|A\|$ denotes $2$-operator norm and  $\|A\|_{\infty}$ is the infinite matrix norm, i.e. $\|A\|_{\infty} = \max_{1 \le i \le n}\sum_{j=1}^n  \vert a_{ij} \vert$. $\mathbb{S}^n$ denotes the set of symmetric matrices and $\mathbb{S}_+^n$ denotes the set of positive semidefinite matrices. $X \succeq 0$ is equivalent to $X \in \mathbb{S}_+^n$. $C_{i, \cdot}$ and $C_{\cdot, j}$ denote the $i$-th row and the $j$-th column of matrix $C$ respectively. The indicator function $\mathcal{I}_{\mathcal{X}}(x)$ of set $\mathcal{X}$ and the limiting-subdifferential $\partial f(x)$ of function $f(x)$ follow definitions in~\cite{Attouch13}. The set $[n]$ denotes $\{1,\ldots,n\}$ and $\lceil{x}\rceil$ is the ceil integer to $x$.

\section{Bilinear Decomposition \& ADMM Algorithm}\label{section-BM-ADMM}
We consider semidefinite programming (SDP) with diagonal constraints in the following form,
\begin{align}
	\begin{aligned}
		\min \quad & \langle C, X \rangle \\
		\text{s.t} \quad & X_{ii} = 1, \text{for } i \in [n] \\
		& \quad X \in \mathbb{S}_+^n,
	\end{aligned} \label{SDP-form}
\end{align}
where $C \in \mathbb{S}^n$. This kind of SDP applies to many applications, including the max-cut problem \cite{Goemans95}, graphical model inference \cite{Erdogdu17} and the community detection problem \cite{Bandeira16}.

It is well known that \eqref{SDP-form} is equivalent to \cite{Wang17,Erdogdu21}:
\begin{align}
	\begin{aligned}
		\min \quad &  f(\sigma): = \langle C, \sigma \sigma^\top \rangle \\
		\text{s.t.} \quad & \| \sigma_i \| = 1, \text{for } i \in [n],
	\end{aligned} \label{sphere-form}
\end{align}
if $r \ge \lceil{\sqrt{2n}}\rceil$, where $\mathbf{\sigma} := [\sigma_1, \sigma_2, ..., \sigma_n]^\top \in \mathbb{R}^{n \times r}$ and $\sigma_i \in \mathbb{R}^r$ is the $i$-th row of $\sigma$. We denote the unit norm constraint set as $\mathcal{M} = \{[\sigma_1, \sigma_2, \dots, \sigma_n]^\top \in \mathbb{R}^{n \times r} \mid \|\sigma_i\|=1, \ i \in [n]\}$. The problem \eqref{sphere-form} is inspired by the popular Burer Monteiro (BM) method, which substitutes the SDP constraint $X \succeq 0$ with a low-rank factorization $X = \sigma \sigma^\top$. The diagonal constraints $X_{ii} = 1, \text{for } i \in [n]$ are replaced by norm constraints $\|\sigma_i\| = 1, \text{for } i \in [n]$. It is easy to see that the constraint set is decoupled while the objective function is coupled w.r.t $\sigma_i, \forall i \in [n]$. In this paper, we introduce a bilinear decomposition for \eqref{sphere-form} as in \cite{Lu17}. Instead of setting $X = \mathbf{\sigma}\mathbf{\sigma}^\top$, we decompose the variable $X$ as a bilinear term, $X = \sigma \tilde{\sigma}^\top$ subject to $\sigma = \tilde{\sigma}$, resulting in the following equivalent problem formulation:
\begin{align}
\begin{aligned}
	\min  & \ \langle C, \tilde{\sigma}\sigma^\top \rangle \\
	\text{s.t.}  & \ \tilde{\sigma} \in \mathcal{M}, \ \tilde{\sigma} = \sigma.
\end{aligned}\label{bilinear-decomposition}
\end{align}
We define the corresponding augmented Lagrangian as:
\begin{align}
	L_{\rho}(\tilde{\sigma}, \sigma, y) := & \langle C, \tilde{\sigma} \sigma^\top \rangle + \langle y, \tilde{\sigma} - \sigma \rangle + \frac{\rho}{2}\|\tilde{\sigma} - \sigma\|^2_F + \sum_{i=1}^n \mathcal{I}_{\|\tilde{\sigma}_i\|=1}(\tilde{\sigma}_i)\notag\\
	= & \sum_{i=1}^n \left[ \frac{\rho}{2} \|\tilde{\sigma}_i - \sigma_i\|^2 + y_i^\top(\tilde{\sigma}_i - \sigma_i) + \mathcal{I}_{\|\tilde{\sigma}_i\|=1}(\tilde{\sigma}_i) \right] + \langle C, \tilde{\sigma} \sigma^\top \rangle,
\end{align}

where $y := [y_1, y_2, \dots, y_n]^\top \in \mathbb{R}^{n \times r}$ is the dual variable for the equality constraint $\tilde{\sigma} = \sigma$ and $ \tilde{\sigma} := [\tilde{\sigma}_1, \tilde{\sigma}_2, ..., \tilde{\sigma}_n]^\top \in \mathbb{R}^{n \times r}$. A standard ADMM algorithm given this splitting is then
\begin{subequations} \label{ADMM-all}
	\begin{align}
		\tilde{\sigma}^{k+1} = & \quad \mathop{\text{argmin}}_{\tilde{\sigma} \in \mathcal{M}} \  L_{\rho}(\tilde{\sigma}, \sigma^k, y^k), \label{ADMM-1}\\
		\sigma^{k+1} = &\quad \mathop{\text{argmin}}_{\sigma} \ L_{\rho}(\tilde{\sigma}^{k+1}, \sigma, y^k), \label{ADMM-2}\\
		y^{k+1} = & \quad y^k + \rho(\tilde{\sigma}^{k+1} - \sigma^{k+1}). \label{ADMM-3}
	\end{align}
\end{subequations}
For the ADMM algorithm above, step \eqref{ADMM-1} corresponds to the solution of $n$ decomposable nonconvex QPs w.r.t. $\tilde{\sigma}_i$ for $ \forall i \in [n]$,
\begin{align}
	\tilde{\sigma}_i^{k+1} & = \mathop{\text{argmin}}_{\|\tilde{\sigma}_i\|=1} \ \frac{\rho}{2} \|\tilde{\sigma}_i - \sigma_i^k\|^2 + {y_i^k}^\top \tilde{\sigma}_i + \sum_{j=1}^n C_{i,j} \langle {\sigma_j^k}, \tilde{\sigma}_i \rangle \label{QP-1}
\end{align}
and has a closed-form solution
\begin{align}
	\tilde{\sigma}^{k+1} \gets \text{Normalize-Row} \left(\sigma^k - \frac{1}{\rho} (y^k + C \sigma^k)\right), \label{normalize-step}
\end{align}
which is to normalize each row vector into unit length. We will require the following assumption to ensure that \eqref{normalize-step} is valid everywhere.
\begin{assumption}\label{assumption-1}
	For $\forall i \in [n]$, $\|\gamma_i^k\|$ is nonzero for any iteration $k$ where $\gamma_i^k$ is the $i$-th row of
	\begin{align}
		\gamma^k := \sigma^k - \frac{1}{\rho} (y^k + C \sigma^k), \label{gamma-update}
	\end{align}
	i.e. $\gamma^k = [\gamma_1^k, \cdots, \gamma_n^k]^\top$.
\end{assumption}
In Lemma~\ref{lemma-gamma-bound}, we will show that Assumption~\ref{assumption-1} is always satisfied for every iteration $k$ given an appropriate choice of $\rho$. In addition, step \eqref{ADMM-2} is an unconstrained QP and amounts to
\begin{align}
	\sigma^{k+1} \gets \tilde{\sigma}^{k+1} + \frac{1}{\rho}(y^k - C\tilde{\sigma}^{k+1}) \label{sigma-update}.
\end{align}

Our approach, outlined in \eqref{ADMM-all}, can therefore be implemented as in Algorithm \ref{alg_ADMM}.
\begin{algorithm}[H]
	\caption{ADMM Burer-Monteiro algorithm (ADMM-BM)}
	\label{alg_ADMM}
	\begin{algorithmic}[1]
		\State
		Initialization: set $\sigma^0 = \tilde{\sigma}^0 \in \mathcal{M},y^0 = C \tilde{\sigma}^0$.
		\\
		\While {\textit{termination criteria not satisfied}}
		\State $\tilde{\sigma}^{k+1} \gets \text{Normalize-Row} \left(\sigma^k - \frac{1}{\rho} (y^k + C \sigma^k)\right)$
		\State $\sigma^{k+1} \gets \tilde{\sigma}^{k+1} + \frac{1}{\rho}(y^k - C\tilde{\sigma}^{k+1})$
		\State $y^{k+1} \gets y^k + \rho(\tilde{\sigma}^{k+1} - \sigma^{k+1})$
		\EndWhile
	\end{algorithmic}
\end{algorithm}
Algorithm \ref{alg_ADMM} has a nice property that connects the primal variable $\tilde{\sigma}$ with the dual variable $y$ for any $k \ge 1$, i.e.\
\begin{align}
	y^k & = y^{k-1} + \rho(\tilde{\sigma}^{k} - \sigma^{k}) = C\tilde{\sigma}^{k},  \label{prim-dual-link}
\end{align}
where the last equality comes from step 3 in Algorithm~\ref{alg_ADMM}.

\section{Convergence of the Objective Value}\label{section-1st-order-convergence-objective}
A first order stationary point of \eqref{sphere-form} can be defined \cite[Def.\ 3.6]{Jiang19} as
\begin{align}
	-(C\tilde{\sigma}^*)_{i, \cdot}^\top & \in \frac{\partial I_{\|\tilde{\sigma}^*_i\|=1}(\tilde{\sigma}^*_i)}{\partial \tilde{\sigma}^*_i} , \ \forall i \in [n], \label{1st-stationary-point}
\end{align}
where $\tilde{\sigma}^*$ belongs to the set of first-order stationary points and $(C\tilde{\sigma}^*)_{i, \cdot}$ denotes the $i$-th row of the matrix $C\tilde{\sigma}^*$.  The main result of this section is then the following theorem:
\begin{theorem} \label{theorem-global}
	If $\rho \ge \max \left\{10\|C\|_{\infty}, 2\cdot \|C\|\right\}$, then the augmented Lagrangian sequence $L_{\rho}(\tilde{\sigma}^k, \sigma^k, y^k)_{k=1}^{+\infty}$ converges to the objective function $ \langle C\tilde{\sigma}^k, \tilde{\sigma}^k \rangle$ in Algorithm \ref{alg_ADMM}, under Assumption \ref{assumption-1}. Moreover, every convergent subsequence of $\{\tilde{\sigma}^k\}$ converges globally to $\Omega$, the set of first-order stationary points of the problem \eqref{sphere-form}, and such a convergent subsequence exists.
\end{theorem}

Theorem~\ref{theorem-global} establishes the global convergence of Algorithm~\ref{alg_ADMM}. The proof of Theorem \ref{theorem-global} is composed of three parts. We first prove monotonic non-increase of the augmented Lagrangian $L_{\rho}(\tilde{\sigma}^k, \sigma^k, y^k)$ in Lemma~\ref{lemma-monotonic-decrease}; we then derive a lower bound for $L_{\rho}(\tilde{\sigma}^k, \sigma^k, y^k), \forall k$ in Lemma~\ref{lemma-lower-obj}. Finally, we use Lemmas \ref{lemma-monotonic-decrease} and \ref{lemma-lower-obj} to prove convergence of $L_{\rho}(\tilde{\sigma}^k, \sigma^k, y^k)$ and $\lim\limits_{k \to \infty}\|\tilde{\sigma}^k - \sigma^k\| \to 0$, which implies convergence to the set of first-order stationary points.

\subsection{Part I: Monotonic non-increase of $L_{\rho}(\tilde{\sigma}, \sigma, y)$}
The key to proving Theorem \ref{theorem-global} is showing that $L_{\rho}(\tilde{\sigma}, \sigma, y)$ is monotonically non-increasing, i.e.
\begin{align}
	L_{\rho}(\tilde{\sigma}^k, \sigma^k, y^k) - L_{\rho}(\tilde{\sigma}^{k+1}, \sigma^{k+1}, y^{k+1}) \ge 0, \quad \forall k. \label{ineq-monotonic-decrease}
\end{align}
In the remainder of this section, we first give the iteration-wise change of $L_{\rho}(\tilde{\sigma}^k, \sigma^k, y^k)$ in Lemma \ref{lemma-decrease} and then prove that the change is non-negative if $\rho$ is properly lower bounded in Lemma \ref{lemma-monotonic-decrease}, which implies \eqref{ineq-monotonic-decrease} is valid.

\noindent\textbf{Change of $L_{\rho}(\tilde{\sigma}^k, \sigma^k, y^k)$}

From \eqref{normalize-step}, we have $\tilde{\sigma}_i^{k+1} = \frac{\gamma_i^k}{\|\gamma_i^k\|}, \forall i \in [n]$ which says $\gamma_i^k$ is aligned with $\tilde{\sigma}_i^{k+1}$ without necessarily being a unit vector itself. The change of  $L_{\rho}(\tilde{\sigma}^k, \sigma^k, y^k)$ per iteration is then summarized in Lemma~\ref{lemma-decrease} below.
\begin{lemma}\label{lemma-decrease}
	For any iteration $k$, we have
	\begin{align}
		& L_{\rho}(\tilde{\sigma}^k, \sigma^k, y^k) - L_{\rho}(\tilde{\sigma}^{k+1}, \sigma^{k+1}, y^{k+1}) \notag \\
		\ge & \left(\frac{\rho \min_{i \in [n]} \{\|\gamma_i^k\|\} }{2} - \frac{\|C\|^2}{\rho} \right)\|\tilde{\sigma}^{k+1} - \tilde{\sigma}^{k}\|^2_F + \frac{\rho}{2}\|\sigma^{k+1} - \sigma^k\|^2_F. \label{al-sequence-diff}
	\end{align}
\end{lemma}

\begin{proof}[Proof of Lemma \ref{lemma-decrease}]
	The proof of \eqref{al-sequence-diff} can be decomposed into three parts
	\begin{align*}
		& L_{\rho}(\tilde{\sigma}^k, \sigma^k, y^k) - L_{\rho}(\tilde{\sigma}^{k+1}, \sigma^{k+1}, y^{k+1})\\
		= & \underbrace{L_{\rho}(\tilde{\sigma}^k, \sigma^k, y^k) - L_{\rho}(\tilde{\sigma}^{k+1}, \sigma^{k}, y^{k})}_{(A)} + \underbrace{L_{\rho}(\tilde{\sigma}^{k+1}, \sigma^k, y^k) - L_{\rho}(\tilde{\sigma}^{k+1}, \sigma^{k+1}, y^{k})}_{(B)} \\
		& + \underbrace{L_{\rho}(\tilde{\sigma}^{k+1}, \sigma^{k+1}, y^k) - L_{\rho}(\tilde{\sigma}^{k+1}, \sigma^{k+1}, y^{k+1})}_{(C)}.
	\end{align*}
	For part $(A)$, we consider the minimization over $\tilde{\sigma}$. The difficulty is the constraint $\|\tilde{\sigma}_i\| = 1, \forall i \in[n]$, which makes the minimization over $\tilde{\sigma}$ nonconvex. We define $L_{\rho,i}(\tilde{\sigma}_i, \sigma, y)$ as
	\begin{align}
		L_{\rho, i}(\tilde{\sigma}_i, \sigma, y) := \frac{\rho}{2} \|\tilde{\sigma}_i - \sigma_i\|^2 + {y_i}^\top (\tilde{\sigma}_i - \sigma_i) + \sum_{j=1}^n C_{i,j} \langle {\sigma_j}, \tilde{\sigma}_i \rangle. \label{min-tilde-x}
	\end{align}
	Note that $L_{\rho, i}(\tilde{\sigma}_i, \sigma^k, y^k)$ in \eqref{min-tilde-x} is differentiable and $\rho$-strongly convex w.r.t. $\tilde{\sigma}_i$. Hence, we have
	\begin{align}
		& L_{\rho, i}(\tilde{\sigma}_i^k, \sigma^k, y^k) \notag\\
		\ge & L_{\rho, i}(\tilde{\sigma}_i^{k+1}, \sigma^k, y^k) + \langle \nabla L_{\rho, i}(\tilde{\sigma}_i^{k+1}, \sigma^k, y^k), \tilde{\sigma}_i^k - \tilde{\sigma}_i^{k+1}\rangle + \frac{\rho}{2}\|\tilde{\sigma}_i^k - \tilde{\sigma}_i^{k+1}\|^2, \notag\\
		= & L_{\rho, i}(\tilde{\sigma}_i^{k+1}, \sigma^k, y^k) + \rho \langle \tilde{\sigma}_i^{k+1} - \gamma_i^k, \tilde{\sigma}_i^k - \tilde{\sigma}_i^{k+1}\rangle + \frac{\rho}{2}\|\tilde{\sigma}_i^k - \tilde{\sigma}_i^{k+1}\|^2, \notag\\
		= & L_{\rho, i}(\tilde{\sigma}_i^{k+1}, \sigma^k, y^k) + \rho (1- \|\gamma_i^k\|) \langle \tilde{\sigma}_i^{k+1}, \tilde{\sigma}_i^k - \tilde{\sigma}_i^{k+1}\rangle + \frac{\rho}{2}\|\tilde{\sigma}_i^k - \tilde{\sigma}_i^{k+1}\|^2, \label{lemma-int-1}
	\end{align}
	where both equalities above come from the definition of $\gamma^k$ in \eqref{gamma-update}. Noting that
	\begin{align*}
		\langle \tilde{\sigma}_i^{k+1}, \tilde{\sigma}_i^k - \tilde{\sigma}_i^{k+1}\rangle & = \langle \tilde{\sigma}_i^{k+1}, \tilde{\sigma}_i^k \rangle - 1 = \langle \tilde{\sigma}_i^{k+1}, \tilde{\sigma}_i^k \rangle - \frac{1}{2}(\|\tilde{\sigma}_i^{k+1}\|^2 + \|\tilde{\sigma}_i^{k}\|^2)\\
		& = - \frac{1}{2}\|\tilde{\sigma}_i^{k+1} - \tilde{\sigma}_i^{k}\|^2,
	\end{align*}
	which uses $\|\tilde{\sigma}_i^{k+1}\| = \|\tilde{\sigma}_i^{k}\| = 1$. Then, \eqref{lemma-int-1} becomes
	\begin{align}
		L_{\rho, i}(\tilde{\sigma}_i^k, \sigma^k, y^k) \ge L_{\rho, i}(\tilde{\sigma}_i^{k+1}, \sigma^k, y^k) + \frac{\rho \|\gamma_i^k\|}{2} \cdot \|\tilde{\sigma}_i^k - \tilde{\sigma}_i^{k+1}\|^2. \label{decrease-1}
	\end{align}
	For part $(B)$, we can establish monotonic non-increase when updating $\sigma$:
	\begin{align}
		& L_{\rho} (\tilde{\sigma}^{k+1}, \sigma^k, y^k) - L_{\rho} (\tilde{\sigma}^{k+1}, \sigma^{k+1}, y^k) \notag\\
		= & \langle \tilde{\sigma}^{k+1}, C(\sigma^k - \sigma^{k+1})\rangle - \langle y^k, \sigma^k - \sigma^{k+1} \rangle + \frac{\rho}{2}\|\tilde{\sigma}^{k+1} - \sigma^k\|^2_F \notag\\
		&- \frac{\rho}{2}\|\tilde{\sigma}^{k+1} - \sigma^{k+1}\|^2_F \notag\\
		= & \langle C\tilde{\sigma}^{k+1} - y^k, \sigma^k - \sigma^{k+1} \rangle + \frac{\rho}{2}\|\tilde{\sigma}^{k+1} - \sigma^k\|^2_F - \frac{\rho}{2}\|\tilde{\sigma}^{k+1} - \sigma^{k+1}\|^2_F \notag\\
		\overset{\eqref{sigma-update}}{=} & \rho \langle \tilde{\sigma}^{k+1} - \sigma^{k+1}, \sigma^k - \sigma^{k+1} \rangle + \frac{\rho}{2}\|\tilde{\sigma}^{k+1} - \sigma^k\|^2_F - \frac{\rho}{2}\|\tilde{\sigma}^{k+1} - \sigma^{k+1}\|^2_F \notag\\
		= & \frac{\rho}{2}\|\tilde{\sigma}^{k+1} - \sigma^k\|^2_F - \frac{\rho}{2} \langle \tilde{\sigma}^{k+1} - \sigma^{k+1}, \tilde{\sigma}^{k+1} - \sigma^k + \sigma^{k+1} - \sigma^k \rangle \notag\\
		= & \frac{\rho}{2}\|\tilde{\sigma}^{k+1} - \sigma^k\|^2_F - \frac{\rho}{2}\left(\|\tilde{\sigma}^{k+1} - \sigma^k\|^2_F - \|\sigma^{k+1} - \sigma^k\|^2_F \right) \notag\\
		= & \frac{\rho}{2}\|\sigma^{k+1} - \sigma^k\|^2_F. \label{decrease-2}
	\end{align}
	Finally, for part $(C)$ we have
	\begin{align}
		L_{\rho} (\tilde{\sigma}^{k+1}, \sigma^{k+1}, y^k) - L_{\rho}  (\tilde{\sigma}^{k+1}, \sigma^{k+1}, y^{k+1}) & = \langle y^k - y^{k+1}, \tilde{\sigma}^{k+1} - \sigma^{k+1} \rangle \notag\\
		& = -\frac{1}{\rho}\|y^{k+1} - y^k\|^2_F,
	\end{align}
	where the last equality comes from the update rule for $y^k$, i.e. step 4 in Algorithm \ref{alg_ADMM}. Furthermore, we can obtain
	\begin{align}
		L_{\rho} (\tilde{\sigma}^{k+1}, \sigma^{k+1}, y^k) - L_{\rho}  (\tilde{\sigma}^{k+1}, \sigma^{k+1}, y^{k+1}) \ge -\frac{1}{\rho}\|C\|^2 \cdot \|(\tilde{\sigma}^{k+1} - \tilde{\sigma}^k)\|^2_F \label{decrease-3}
	\end{align}
	from \eqref{prim-dual-link}. Combining \eqref{decrease-1}, \eqref{decrease-2} and \eqref{decrease-3}, we finally obtain \eqref{al-sequence-diff}.
\end{proof}
\vspace{3mm}
\noindent \textbf{Choice of $\rho$}\\
From Lemma~\ref{lemma-decrease}, it is sufficient to show that $ L_{\rho}(\tilde{\sigma}^k, \sigma^k, y^k)$ is monotonically non-increasing in \eqref{ineq-monotonic-decrease} if the coefficient $\frac{\rho \min_{i \in [n]} \{\|\gamma_i^k\|\} }{2} -~\frac{\|C\|^2}{\rho} > 0$ while $\|\gamma_i^k\|$ is dependent on the choice of row $i$ and iteration $k$. We argue that by choosing $\rho$ properly, $\|\gamma_i^k\|$ can be uniformly lower-bounded in Lemma \ref{lemma-gamma-bound} and then $L_{\rho}(\tilde{\sigma}^k, \sigma^k, y^k)$ always satisfies the monotonic non-increase condition \eqref{ineq-monotonic-decrease} in Lemma \ref{lemma-monotonic-decrease}.

From \eqref{gamma-update}, we expect that the magnitude of $\|\gamma_i^k\|$ will depend strongly on the choice of $\rho$, and $\|\gamma_i^k\| \to 1, \ \forall i \in [n]$ if $\rho$ is set to be sufficiently large as $\sigma \to \tilde{\sigma}$ with $\|\tilde{\sigma}_i\| = 1, \ \forall i \in [n]$. In Lemma \ref{lemma-gamma-bound}, we provide a lower bound for $\rho$ such that $\|\gamma_i^k\|, \forall i \in [n]$ is not too small for any $k$.
\begin{lemma}\label{lemma-gamma-bound}
	Suppose $\rho \ge \alpha \|C\|_{\infty}, \alpha>0$, then
	\begin{align*}
		\|\gamma_i^k\| \ge 1 - \frac{4}{\alpha} - \frac{2}{\alpha^2}, \ \forall i, \ \forall k \ge 2.
	\end{align*}
\end{lemma}
\begin{proof}[Proof of Lemma \ref{lemma-gamma-bound}]
	We note that, $\forall i \in [n]$,~$\forall k \ge 2$,
	\begin{align*}
		\sigma_i^{k+1} = \tilde{\sigma}_i^{k+1} + \frac{1}{\rho}(y_i^k - \sum_{j=1}^n C_{i,j}\tilde{\sigma}^{k+1}_j) \overset{\eqref{prim-dual-link}}{=}  \tilde{\sigma}_i^{k+1} + \frac{1}{\rho}\sum_{j=1}^n C_{i,j} ( \tilde{\sigma}^k_j - \tilde{\sigma}^{k+1}_j).
	\end{align*}
	Substituting into \eqref{gamma-update}, we can write $\gamma_i^k$ in terms of $\tilde{\sigma}$ as
	\begin{align*}
		\gamma_i^k & = \sigma_i^k - \frac{1}{\rho} (y_i^k +  \sum_{j=1}^n C_{i,j} \sigma^{k}_j) \overset{\eqref{prim-dual-link}}{=} \sigma_i^k - \frac{1}{\rho} \sum_{j=1}^n C_{i,j}(\tilde{\sigma}^{k}_j + \sigma^{k}_j)\\
		& = \tilde{\sigma}_i^{k} + \frac{1}{\rho}\sum_{j=1}^n C_{i,j} ( \tilde{\sigma}^{k-1}_j - \tilde{\sigma}^{k}_j)  - \frac{1}{\rho} \sum_{j=1}^n C_{i,j}\tilde{\sigma}^{k}_j\\
		& \quad - \frac{1}{\rho} \sum_{j=1}^n C_{i,j} \left[\tilde{\sigma}_j^{k} + \frac{1}{\rho}\sum_{l=1}^n C_{j,l} ( \tilde{\sigma}^{k-1}_l - \tilde{\sigma}^{k}_l)\right],
	\end{align*}
	where the last equality is based on \eqref{sigma-update} and \eqref{prim-dual-link}, the update of $\sigma^k$. Since $\|\tilde{\sigma}^k_i\| = 1, \ \forall i\in [n], \forall k \ge 1$ from step 2 of Algorithm~\ref{alg_ADMM}, we can bound the norm of $\|\gamma_i^k\|$ by
	\begin{align*}
		\begin{aligned}
			\|\gamma_i^k\| & \ge 1- \frac{4}{\rho}\sum_{j=1}^n \vert C_{i,j} \vert - \frac{1}{\rho}\sum_{j=1}^n \vert C_{i,j} \vert \cdot \left(\frac{2}{\rho}\sum_{l=1}^n \vert C_{j,l} \vert \right)\\
			& \ge 1- \frac{4}{\rho}\sum_{j=1}^n \vert C_{i,j} \vert - \frac{1}{\rho}\sum_{j=1}^n \vert C_{i,j} \vert \cdot \frac{2}{\alpha}\\
			& \ge 1 - \frac{4}{\alpha} - \frac{2}{\alpha^2},
		\end{aligned}
	\end{align*}
	where the second inequality relies on $\rho \ge \alpha \|C\|_{\infty}$.
\end{proof}
Note that $C$ is in the linear objective function and can be scaled such that the parameter $\rho$ won't be too large to satisfy the condition in Lemma \ref{lemma-gamma-bound}. Meanwhile, Assumption~\ref{assumption-1} will be satisfied automatically when $k \ge 2$, as long as $1 - \frac{4}{\alpha} - \frac{2}{\alpha^2} > 0$.

Based on \eqref{al-sequence-diff} and Lemma \ref{lemma-gamma-bound}, we establish that $L_{\rho}(\tilde{\sigma}, \sigma, y)$  is monotonically non-increasing given a proper $\rho$ in Lemma \ref{lemma-monotonic-decrease}.
\begin{lemma} \label{lemma-monotonic-decrease}
	If we set $\rho \ge \max \left\{\alpha \|C\|_{\infty}, \beta \|C\|\right\}$ and $\alpha>0, \beta>0$ satisfies $\kappa := \frac{(\alpha^2-4\alpha-2)\beta}{2\alpha^2} - \frac{1}{\beta} > 0$, then the augmented Lagrangian sequence $L_{\rho}(\tilde{\sigma}^k, \sigma^k, y^k)$ is monotonically non-increasing $\forall k \ge 2$, and satisfies
	\begin{align*}
		L_{\rho}(\tilde{\sigma}^k, \sigma^k, y^k) - L_{\rho}(\tilde{\sigma}^{k+1}, \sigma^{k+1}, y^{k+1})
		\ge & \kappa \|C\| \cdot \|\tilde{\sigma}^{k+1} - \tilde{\sigma}^{k}\|^2_F + \frac{\rho}{2}\|\sigma^{k+1} - \sigma^k\|^2_F.
	\end{align*}
\end{lemma}
\begin{proof}[Proof of Lemma \ref{lemma-monotonic-decrease}]
	First, Lemma \ref{lemma-gamma-bound} is satisfied under the condition $\rho \ge \alpha\|C\|_{\infty}$. Then, combining Lemma \ref{lemma-gamma-bound} with \eqref{al-sequence-diff} we obtain
	\begin{align*}
		& L_{\rho}(\tilde{\sigma}^k, \sigma^k, y^k) - L_{\rho}(\tilde{\sigma}^{k+1}, \sigma^{k+1}, y^{k+1}) \\
		\ge & \left[\frac{\alpha^2  - 4\alpha - 2}{2\alpha^2}  \cdot  \rho - \frac{\|C\|^2}{\rho} \right]  	\|\tilde{\sigma}^{k+1}  -   \tilde{\sigma}^{k}\|^2_F +  \frac{\rho}{2}\|\sigma^{k+1} - \sigma^k\|^2_F\\
		\ge & \kappa \|C\| \cdot \|\tilde{\sigma}^{k+1} - \tilde{\sigma}^{k}\|^2_F + \frac{\rho}{2}\|\sigma^{k+1} - 	\sigma^k\|^2_F.
	\end{align*}
\end{proof}
\subsection{Part II: Lower bound of $L_{\rho} (\tilde{\sigma}^k, \sigma^k, y^k)$}
We next establish a lower bound on $\lim\limits_{k \to \infty} L_{\rho} (\tilde{\sigma}^k, \sigma^k, y^k)$ in the following lemma:
\begin{lemma} \label{lemma-lower-obj}
	For $k \ge 2$, we have
	\begin{align}
		L_{\rho}(\tilde{\sigma}^k, \sigma^k, y^k) \ge -n\|C\|_{\infty}. \label{equivalent-aug-lagrangian}
	\end{align}
\end{lemma}
\begin{proof}[Proof of Lemma \ref{lemma-lower-obj}]
	\begin{align}
		L(\tilde{\sigma}^k, \sigma^k, y^k) = & \langle C \tilde{\sigma}^k, \sigma^k \rangle + \langle y^k,\tilde{\sigma}^k - \sigma^k \rangle + \frac{\rho}{2}\|\tilde{\sigma}^k - \sigma^k\|^2_F \notag \\
		= & \langle C \tilde{\sigma}^k, \tilde{\sigma}^k \rangle + \langle y^k - C \tilde{\sigma}^k,\tilde{\sigma}^k - \sigma^k \rangle + \frac{\rho}{2}\|\tilde{\sigma}^k - \sigma^k\|^2_F \notag \\
		\overset{\eqref{prim-dual-link}}{=} & \langle C \tilde{\sigma}^k, \tilde{\sigma}^k \rangle + \frac{\rho}{2}\|\tilde{\sigma}^k - \sigma^k\|^2_F \notag \\
		\ge & \langle C\tilde{\sigma}^k,\tilde{\sigma}^k \rangle = \sum_{i=1}^{n}\sum_{j=1}^{n} C_{i,j} \langle \tilde{\sigma}^k_i, \tilde{\sigma}^k_j \rangle \overset{\|\tilde{\sigma}^k_i\| = 1}{\ge} -\sum_{i=1}^{n}\sum_{j=1}^{n} \vert C_{i,j} \vert \notag\\
		\ge & -n\|C\|_{\infty}. \notag
	\end{align}
\end{proof}

\subsection{Proof of Theorem \ref{theorem-global}}
We can now prove Theorem~\ref{theorem-global} given the results of Lemma~\ref{lemma-monotonic-decrease} and Lemma \ref{lemma-lower-obj}.
\begin{proof}[Proof of Theorem \ref{theorem-global}]
	Since $L_{\rho} (\tilde{\sigma}^k, \sigma^k, y^k)$ is monotonically non-increasing (Lemma \ref{lemma-monotonic-decrease}) and lower bounded (Lemma \ref{lemma-lower-obj}), $L_{\rho} (\tilde{\sigma}^k,  \sigma^k,  y^k)$ converges to a constant value due to the monotone convergence theorem. As a consequence, both $\|\sigma^{k+1} - \sigma^k\|_F$ and $ \|\tilde{\sigma}^{k+1} - \tilde{\sigma}^k\|_F$ converge to $0$ from Lemma \ref{lemma-monotonic-decrease}. We also have \begin{align*}
		\|y^{k+1} - y^k\|_F \le  \|C\|_F \|(\tilde{\sigma}^{k+1} - \tilde{\sigma}^k)\|_F,
	\end{align*}
	due to \eqref{prim-dual-link}. Hence, $\|y^{k+1} - y^k\|_F$ also converges to $0$ when $k$ is sufficiently large. In addition, $\|y^{k+1} - y^k\|_F$ can be written as
	\begin{align*}
		\|y^{k+1} - y^k\|_F = \rho\|\tilde{\sigma}^{k+1} - \sigma^{k+1}\|_F
	\end{align*}
	due to the dual update \eqref{ADMM-3}. Since $\|y^{k+1} - y^k\|_F \to 0$, we obtain $\|\tilde{\sigma}^{k+1} - \sigma^{k+1}\|_F \rightarrow 0$ for constant $\rho$. According to \eqref{equivalent-aug-lagrangian}, we have
	\begin{align*}
		L_{\rho} (\tilde{\sigma}^k, \sigma^k, y^k) = \langle C \tilde{\sigma}^k, \tilde{\sigma}^k \rangle + \frac{\rho}{2}\|\tilde{\sigma}^k - \sigma^k\|^2_F.
	\end{align*}
	Due to the convergence of $\|\tilde{\sigma}^k - \sigma^k\|_F \rightarrow 0$, we obtain
	\begin{align*}
		L_{\rho} (\tilde{\sigma}^k, \sigma^k, y^k) \rightarrow \langle C\tilde{\sigma}^k, \tilde{\sigma}^k \rangle.
	\end{align*}
	Moreover, the optimality condition of \eqref{ADMM-1} can be shown to imply that
	\begin{align}
		0 \in (C\sigma^k)_{i, \cdot}^\top + y^k_i +\rho (\tilde{\sigma}_i^{k+1} - \sigma_i^k) + \frac{\partial I_{\|\tilde{\sigma}_i\|=1}(\tilde{\sigma}_i^{k+1})}{\partial \tilde{\sigma}_i^{k+1}},
	\end{align}
	following the same argument as in~\cite{Attouch13}. Since $\|\tilde{\sigma}_i^{k+1} - \tilde{\sigma}_i^k\| \rightarrow 0$ and $\|\tilde{\sigma}_i^k - \sigma_i^k\| \rightarrow 0$, we obtain
	\begin{align}
		\lim\limits_{k \to +\infty} \text{dist}\left(- (C\sigma^k)_{i, \cdot}^\top - y^k_i, \frac{\partial I_{\|\tilde{\sigma}_i\|=1}(\tilde{\sigma}_i^{k+1})}{\partial \tilde{\sigma}_i^{k+1}} \right) = 0
	\end{align}
	and hence
	\begin{align}
		\lim\limits_{k \to +\infty}\text{dist}\left(- 2(C\tilde{\sigma}^k)_{i, \cdot}^\top, \frac{\partial I_{\|\tilde{\sigma}_i\|=1}(\tilde{\sigma}_i^k)}{\partial \tilde{\sigma}_i^k} \right) = 0
	\end{align}
	due to $\|\sigma^{k+1} - \sigma^k\|_F \to 0$, $\|\tilde{\sigma}^{k} - \sigma^k\|_F \rightarrow 0$ and~\eqref{prim-dual-link}, which means $\tilde{\sigma}^k$ lies in the set of first-order stationary points as defined in \eqref{1st-stationary-point} when $k \to \infty$. In addition, the compactness of $\mathcal{M}$ implies that there is a subsequence of $\tilde{\sigma}^k$ that converges to a first-order stationary point.
\end{proof}

\section{Convergence of the Sequence $\tilde{\sigma}^k$}\label{section-1st-convergence-sequence}
Theorem~\ref{theorem-global} established convergence of the sequence of Langrangian values $L_\rho(\tilde \sigma ^k,\sigma^k,y^k)$.   We next consider the behavior of the sequence of iterates $\tilde{\sigma}^k$.  We show that this sequence converges to a first-order stationary point in the following Theorem, which is the main result of this section:
\begin{theorem} \label{theorem-global-1st}
	Under Assumption \ref{assumption-1}, if we set $\rho \ge \max \left\{10\|C\|_{\infty}, 2\cdot \|C\|\right\}$, then the sequence $(\tilde{\sigma}^k)_{k \in \mathbb{N}}$ generated by Algorithm~\ref{alg_ADMM} converges to a critical point $\bar{\sigma}$ of $f$. Moreover, the sequence $(\tilde{\sigma}^k)_{k \in \mathbb{N}}$ has a finite length, i.e. $\sum_{k=0}^{+\infty} \|\tilde{\sigma}^{k+1} - \tilde{\sigma}^k\| < \infty$.
\end{theorem}
Our result relies on the Kurdyka-\L ojasiewicz property, defined as follows:
\begin{definition}[Kurdyka-\L ojasiewicz (KL) property]\label{D:Kurdyk}
	The function $f$ is said to have the Kurdyka-\L ojasie\-wicz property at $\bar{x} \in \mbox{\rm dom} \ \! \partial \! f$ if there exist $\eta \in (0,+\infty]$, a neighborhood $U$ of $\bar{x}$ and a continuous concave function $\varphi:[0,\eta)\rightarrow \mathbb{R}_+$ such that:

	{\bf -} $\varphi(0)=0$,

	{\bf -} $\varphi$ is $C^1$ on $(0,\eta)$,

	{\bf -} for all $s\in (0,\eta)$, $\varphi'(s)>0$,

	{\bf -} and for all $x$ in $U \cap[f(\bar x)<f<f(\bar x)+\eta]$, the
	Kurdyka-\L ojasie\-wicz inequality holds
	\begin{equation*}
		\varphi'(f(x)-f(\bar{x}))\,\mbox{\rm dist}(0,\partial f(x)) \geq 1.
	\end{equation*}
	Moreover, $f$ is called a KL function if it satisfies KL property at each point of $\text{dom} \ \! \partial \! f$.
\end{definition}
When a function $f$ is known to be a KL function, we can prove the convergence of an algorithm to critical points by checking conditions in the following proposition from~\cite{Attouch13}:
\begin{lemma}[Theorem 2.9 in~\cite{Attouch13}]\label{lemma-KL-convergence}
	Suppose $h: \mathbb{R}^n \to \mathbb{R} \cup \{\infty\}$ is a proper lower semicontinuous function and $(x^{k})_{k \in \mathbb{N}}$ is a sequence satisfying all of the following:
	\begin{itemize}
		\item (C1) \textit{Sufficient decrease condition}: $\forall k \in \mathbb{N},$
		\begin{align*}
			h(x^{k+1}) + a\|x^{k+1} - x^k\|^2 \le h(x^{k});
		\end{align*}
		\item (C2) \textit{Relative error condition}: $\forall k \in \mathbb{N}, \exists \ g^{k+1} \in \partial h(x^{k+1})$ such that
		\begin{align*}
			\|g^{k+1}\| \le b \|x^{k+1} - x^k\|;
		\end{align*}
		\item (C3) \textit{Continuity condition}: There exists a subsequence $(x^{k_j})_{j \in \mathbb{N}}$ and $\bar{x}$ such that, when $j \to \infty$,
		\begin{align*}
			x^{k_j} \to \bar{x} \text{ and } h(x^{k_j}) \to h(\bar{x}).
		\end{align*}
	\end{itemize}
	Here, $a, b$ are positive constant. If $h$ has the Kurdyka-\L ojasie\-wicz property at the cluster point $\bar{x}$ specified in (C3), then $\bar{x}$ is a critical point of $h$ and the sequence $(x^k)_{k \in \mathbb{N}}$ converges to it as $k \to \infty$. Moreover, the sequence $(x^k)_{k \in \mathbb{N}}$ has a finite length, i.e.
	\begin{align*}
		\sum_{k=0}^{+\infty}\|x^{k+1} - x^k\| < \infty.
	\end{align*}
\end{lemma}
The lemma above is the key component to prove Theorem~\ref{theorem-global-1st}.

Next, we discuss how to utilize the KL property inside the proof of Theorem~\ref{theorem-global-1st}. The proof of Theorem~\ref{theorem-global-1st} follows \cite{Bolte14},~\cite{Attouch13} and proceeds in three parts. We first define a \textit{twin problem} in Section~\ref{subsubsection-twin-problem} and prove the equivalence of first-order stationary points between the original problem and the twin one. Next, we show Algorithm~\ref{alg_ADMM} converges to critical points of the twin problem via Lemma~\ref{lemma-KL-convergence} and thus also critical points of the original problem~\eqref{sphere-form} by Lemma~\ref{lemma-critical-point}. The key is to show that the twin problem satisfies (C1), (C2), (C3) and $G_{\rho}(\tilde{\sigma}, \sigma)$ is a KL function.

\subsection{Twin problem}\label{subsubsection-twin-problem}
Since we have obtained $y^k = C\tilde{\sigma}^k, k \ge 1$ in~\eqref{prim-dual-link}, Algorithm~\ref{alg_ADMM} can also be interpreted as an alternating method for the minimization over the function
\[
G_{\rho}(\tilde{\sigma}, \sigma) := \langle C, \tilde{\sigma} \tilde{\sigma}^\top \rangle + \frac{\rho}{2}\|\tilde{\sigma} - \sigma\|^2_F + \sum_{i=1}^n \mathcal{I}_{\|\tilde{\sigma}_i\|=1}(\tilde{\sigma}_i) \label{def-Grho}
\]
Note that $G_{\rho}(\tilde{\sigma}^k, \sigma^k) = L_{\rho}(\tilde{\sigma}^k, \sigma^k, y^k)$ for all $k \ge 1$  when we take $y = C\tilde{\sigma}$ in the augmented Lagrangian $L_{\rho}(\mathbf{\sigma}, \tilde{\mathbf{\sigma}}, y)$. It then remains to prove Algorithm~\ref{alg_ADMM} converges to a critical point of~\eqref{def-Grho}, which is also a critical point of~\eqref{sphere-form} due to the following lemma.
\begin{lemma}\label{lemma-critical-point}
	The problem~\eqref{def-Grho} and the problem~\eqref{sphere-form} have the same critical points.
\end{lemma}
\begin{proof}[Proof of Lemma~\ref{lemma-critical-point}]

	A critical point of~\eqref{def-Grho} satisfies
	\begin{align*}
		\partial_{\tilde{\sigma}} G_{\rho}(\tilde{\sigma}, \sigma) & = 2C\tilde{\sigma} + \rho(\tilde{\sigma}-\sigma) + \partial \left(\sum_{i=1}^n \mathcal{I}_{\|\tilde{\sigma}_i\|=1}(\tilde{\sigma}_i)\right), \\
		\partial_{\sigma} G_{\rho}(\tilde{\sigma}, \sigma) & = \rho(\sigma - \tilde{\sigma}).\\
	\end{align*}
	$0 \in \partial_{\tilde{\sigma}} G_{\rho}(\tilde{\sigma}, \sigma)$ and $\partial_{\sigma} G_{\rho}(\tilde{\sigma}, \sigma) = 0$ (a critical point to $G_{\rho}(\tilde{\sigma}, \sigma)$) are equivalent to $\tilde{\sigma}$ is a critical point to the problem~\eqref{sphere-form}.
\end{proof}

\subsection{Proof of Theorem~\ref{theorem-global-1st}}
In order to prove Theorem~\ref{theorem-global-1st} we will first need the following result showing that the property (C2) is satisfied for $G_{\rho}(\tilde{\sigma}, \sigma)$:
\begin{lemma}\label{lemma-C2}
	The sequence $(\tilde{\sigma}^k, \sigma^k)$ generated by Algorithm~\ref{alg_ADMM} satisfies the property (C2) w.r.t. $G_{\rho}(\tilde{\sigma}, \sigma)$:
	\begin{align}
		\bigg{\|}\begin{matrix}
			\partial_{\tilde{\sigma}}G_{\rho}(\tilde{\sigma}^{k+1}, \sigma^{k+1})\\
			\partial_{\sigma}G_{\rho}(\tilde{\sigma}^{k+1}, \sigma^{k+1})
		\end{matrix}\bigg{\|}_F
		\le \left( 2\|C\|+\rho+\frac{\|C\|^2}{\rho} \right)
		\bigg{\|}\begin{matrix}
			\tilde{\sigma}^{k+1} - \tilde{\sigma}^k\\
			\sigma^{k+1} - \sigma^k
		\end{matrix}\bigg{\|}_F, \ \forall k \ge 1 .
	\end{align}
\end{lemma}
\begin{proof}
	According to Algorithm~\ref{alg_ADMM}, we have
	\begin{align}
		C\sigma^k + y^k + \rho(\tilde{\sigma}^{k+1} - \sigma^k) + v^{k+1} = 0, \label{1st-tilde-sigma}\\
		C\tilde{\sigma}^{k+1} - y^k + \rho(\sigma^{k+1} - \tilde{\sigma}^{k+1}) = 0,
	\end{align}
	due to the first-order optimality condition of~\eqref{ADMM-1},~\eqref{ADMM-2} where
	\begin{align*}
		v^{k+1} = \sum_{i=1}^n \partial I_{\|\tilde{\sigma}_i\|=1}(\tilde{\sigma}_i^{k+1}).
	\end{align*}
	Then, a subgradient of $G_{\rho}(\tilde{\sigma}, \sigma)$ at $(\tilde{\sigma}^{k+1}, \sigma^{k+1})$ is
	\begin{align}
		\partial_{\tilde{\sigma}}G_{\rho}(\tilde{\sigma}^{k+1}, \sigma^{k+1}) = & 2C\tilde{\sigma}^{k+1} + \rho(\tilde{\sigma}^{k+1} - \sigma^{k+1}) + v^{k+1} \notag\\
		\overset{\eqref{1st-tilde-sigma}}{=} & 2C\tilde{\sigma}^{k+1} + \rho(\tilde{\sigma}^{k+1} - \sigma^{k+1}) - [C\sigma^k + y^k + \rho(\tilde{\sigma}^{k+1} - \sigma^k)], \notag\\
		\overset{\eqref{prim-dual-link}}{=} & 2C\tilde{\sigma}^{k+1} - C\sigma^k - C\tilde{\sigma}^k - \rho(\sigma^{k+1} - \sigma^{k}) \notag\\
		= & C(\tilde{\sigma}^{k+1} - \tilde{\sigma}^k) + C(\sigma^{k+1} - \sigma^k) +C(\tilde{\sigma}^{k+1} - \sigma^{k+1}) - \rho(\sigma^{k+1} - \sigma^{k}) \notag\\
		\overset{\eqref{ADMM-3}}{=} & C(\tilde{\sigma}^{k+1} - \tilde{\sigma}^k) + C(\sigma^{k+1} - \sigma^k) +\frac{C}{\rho}(y^{k+1}-y^k) - \rho(\sigma^{k+1} - \sigma^{k}) \notag\\
		\overset{\eqref{prim-dual-link}}{=} & C(\tilde{\sigma}^{k+1} - \tilde{\sigma}^k) + C(\sigma^{k+1} - \sigma^k) +\frac{C^2}{\rho}(\tilde{\sigma}^{k+1}-\tilde{\sigma}^k) - \rho(\sigma^{k+1} - \sigma^{k})\\
		\partial_{\sigma}G_{\rho}(\tilde{\sigma}^{k+1}, \sigma^{k+1}) = & \rho(\sigma^{k+1} - \tilde{\sigma}^{k+1}) \overset{\eqref{ADMM-3}}{=} y^k - y^{k+1} \overset{\eqref{prim-dual-link}}{=} C(\tilde{\sigma}^k - \tilde{\sigma}^{k+1}).
	\end{align}
	Therefore,
	\begin{align}
		\bigg{\|}\begin{matrix}
			\partial_{\tilde{\sigma}}G_{\rho}(\tilde{\sigma}^{k+1}, \sigma^{k+1})\\
			\partial_{\sigma}G_{\rho}(\tilde{\sigma}^{k+1}, \sigma^{k+1})
		\end{matrix}\bigg{\|}_F
		\le \left( 2\|C\|+\rho+\frac{\|C\|^2}{\rho} \right)
		\bigg{\|}\begin{matrix}
			\tilde{\sigma}^{k+1} - \tilde{\sigma}^k\\
			\sigma^{k+1} - \sigma^k
		\end{matrix}\bigg{\|}_F, \ \forall k \ge 1 .
	\end{align}
\end{proof}

Finally, we give the proof for Theorem~\ref{theorem-global-1st}.
\begin{proof}[Proof of Theorem~\ref{theorem-global-1st}]
	(C1) is valid for $G_{\rho}(\tilde{\sigma}, \sigma)$ with $a = \min \{\kappa\|C\|, \rho\}$ since we can rewrite Lemma~\ref{lemma-monotonic-decrease} as
	\begin{align}
		G_{\rho}(\tilde{\sigma}^k, \sigma^k) - G_{\rho}(\tilde{\sigma}^{k+1}, \sigma^{k+1})
		\ge & \kappa \|C\| \cdot \|\tilde{\sigma}^{k+1} - \tilde{\sigma}^{k}\|^2_F + \frac{\rho}{2}\|\sigma^{k+1} - \sigma^k\|^2_F,
	\end{align}
	based on the primal-dual connection~\eqref{prim-dual-link}. (C2) is validated in Lemma~\ref{lemma-C2}. Similar to the proof of Theorem~\ref{theorem-global}, $\tilde{\sigma}^k \in \mathcal{M}$ and the update of $\sigma^k$ in Algorithm~\ref{alg_ADMM} imply that the sequence $(\tilde{\sigma}^k, \sigma^k)$ is bounded. Therefore, there exists a converging subsequence $(\tilde{\sigma}^{k_j}, \sigma^{k_j})$ and the continuity of $G_{\rho}(\tilde{\sigma}, \sigma)$ on $\mathcal{M} \times \mathbb{R}^{n \times r}$ mean that (C3) is satisfied. In addition, $G_{\rho}(\tilde{\sigma}, \sigma)$ is semi-algebraic and thus a KL function that satisfies the KL property at any point of $\mbox{\rm dom} \ \! G_{\rho}(\tilde{\sigma}, \sigma)${\footnote{The semi-algebraic functions and related KL property are detailed in \cite[\S5]{Bolte14}.}}. Therefore, Lemma~\ref{lemma-KL-convergence} applies to $G_{\rho}(\tilde{\sigma}, \sigma)$ and the sequence $(\tilde{\sigma}, \sigma)$ generated by Algorithm~\ref{alg_ADMM} converges to a critical point of the problem~\eqref{def-Grho} and thus a critical point of \eqref{sphere-form} due to Lemma~\ref{lemma-critical-point}, with
	\begin{align*}
		\sum_{k=0}^{+\infty} \left(\|\tilde{\sigma}^{k+1} - \tilde{\sigma}^k\| + \|\sigma^{k+1} - \sigma^k\|\right) < \infty,
	\end{align*}
	which implies $\sum_{k=0}^{+\infty} \|\tilde{\sigma}^{k+1} - \tilde{\sigma}^k\| < \infty$.
\end{proof}

\section{Exploitation of Negative Curvature}\label{section-negative-curvature}
Thus far we have analysed Algorithm~\ref{alg_ADMM} in the Euclidean metric, but it is known that SDPs can be regarded as optimization over manifolds~\cite{Absil07,Boumal20}. For a class of SDP problems including diagonally constrained SDPs, any local optimum is also globally optimal in manifold optimization~\cite{Boumal16}. Hence, we can exploit second-order information and add it into Algorithm~\ref{alg_ADMM} to achieve local optimality, and thus global optimal. In this section, we analyse the convergence of Algorithm~\ref{alg_ADMM} in view of manifold optimization.

\subsection{Convergence to a first-order stationary point on manifold}
We first prove the following lemma:
\begin{lemma}\label{lemma-eq-critical-point}
	Any first-order stationary point (critical point) defined as in \eqref{1st-stationary-point} is a first-order stationary point (critical point) of the Cartesian products of $n$ spherical manifolds.
\end{lemma}
\begin{proof}[Proof of Lemma~\ref{lemma-eq-critical-point}]
	Note that the manifold gradient of Problem \eqref{sphere-form} is
	\begin{align}
		\text{grad}f(\sigma) = 2(C - \text{Diag}(\text{diag}(C\sigma \sigma^\top)))\sigma, \ \forall \sigma \in \mathcal{M}.
	\end{align}
	Any first-order stationary point $\sigma^*$ defined by \eqref{1st-stationary-point} is equivalent to
	\begin{align}
		C \sigma^* = \Lambda \sigma^*, \sigma^* \in \mathcal{M},
	\end{align}
	where $\Lambda = \text{Diag}(\lambda_1, \cdots, \lambda_n), \lambda_i \in \mathbb{R}, \forall i \in [n]$, and we obtain
	\begin{align}
		\begin{aligned}
			\text{grad}f(\sigma^*) & = 2(C - \text{Diag}(\text{diag}(C\sigma^* {\sigma^*}^\top)))\sigma^* \\
			& = 2(\Lambda - \text{Diag}(\text{diag}(\Lambda \sigma^* {\sigma^*}^\top)))\sigma^*\\
			& = \textbf{0},
		\end{aligned}
	\end{align}
	where the last equality is due to the property of $\mathcal{M}$ that $\|\sigma_i\|=1, \forall i \in [n]$.
\end{proof}

\subsection{Achieving $O(1-1/r)$ optimality with negative curvature}
For nonconvex optimization, a first-order algorithm may stall at a saddle point. We can improve the convergence of our algorithm to second order stationary points by exploiting negative curvature when it is close to a saddle point. First, we define approximate convex points of a function $f$ on manifold $\mathcal{M}$.
\begin{definition}[Approximate convex point]
	Let $f$ be a twice differentiable function on a Riemannian manifold $\mathcal{M}$. The point $\sigma \in \mathcal{M}$ is an $\epsilon$-approximate convex point of $f$ on $\mathcal{M}$ if
	\begin{equation*}
		\langle u, \text{Hess} f(\sigma)[u] \rangle \ge -\epsilon \langle u, u\rangle, \quad \forall u \in T_{\sigma} \mathcal{M},
	\end{equation*}
	where $\text{Hess} f(\sigma)$ denotes the Riemannian Hessian of $f$ at point $\sigma$ and $\langle \cdot,\cdot \rangle$ is the scalar product on $T_{\sigma} \mathcal{M}$.
\end{definition}

We can now extend Algorithm~\ref{alg_ADMM} to exploit negative curvature, resulting in Algorithm~\ref{alg_ADMM2}.
\begin{algorithm}[H]
	\caption{ADMM-BM2}
	\label{alg_ADMM2}
	\begin{algorithmic}[1]
		\State Initialization: set $\sigma^0 = \tilde{\sigma}^0 \in \mathcal{M},y^0 = C \tilde{\sigma}^0, \epsilon > 0$.
		\\
		\For{$k=0,1,2, \dots$}
		\State $\tilde{\sigma}^{k+1} \gets \text{Normalize-Row} \left(\sigma^k - \frac{1}{\rho} (y^k + C \sigma^k)\right)$
		\State $\sigma^{k+1} \gets \tilde{\sigma}^{k+1} + \frac{1}{\rho}(y^k - C\tilde{\sigma}^{k+1})$
		\State $y^{k+1} \gets y^k + \rho(\tilde{\sigma}^{k+1} - \sigma^{k+1})$
		\If{$G_{\rho}(\tilde{\sigma}^{k}, \sigma^{k}) - G_{\rho}(\tilde{\sigma}^{k+1}, \sigma^{k+1}) \ge \Delta$}
		\State Continue with $\tilde{\sigma}^{k+1}, \sigma^{k+1}, y^{k+1}$.
		\Else
		\State Find $u^k \in T_{\tilde{\sigma}^k}\mathcal{M}$ such that $\lambda_{H}(\tilde{\sigma}^k) := \langle u^k, \text{Hess} f(\tilde{\sigma}^k)[u^k] \rangle \le \lambda_{\min}(\text{Hess} f(\tilde{\sigma}^k))/2$ with $\langle u^k, \text{grad} f(\tilde{\sigma}^k) \rangle \le 0$ and $\|u^k\|_F=1$.
		\If{$\lambda_{H}(\tilde{\sigma}^k)  < -\frac{\epsilon}{2}$}
		\State $\tilde{\sigma}_i^{k+1} \gets \tilde{\sigma}_i^k \cos(\|u_i^k\|t) + \frac{u_i^k}{\|u_i^k\|}\sin(\|u_i^k\|t), \forall i \in [n]$ with $t = -\frac{2}{15\|C\|_1}\lambda_{H}(\tilde{\sigma}^k)$.
		\State $\sigma^{k+1} \gets \tilde{\sigma}^{k+1}$
		\State $y^{k+1} \gets C \tilde{\sigma}^{k+1}$
		\Else
		\State Return an $\epsilon$-approximate convex point with high probability.
		\EndIf
		\EndIf
		\EndFor
	\end{algorithmic}
\end{algorithm}
Compared with Algorithm~\ref{alg_ADMM}, we additionally check the consecutive difference of $G_{\rho}(\tilde{\sigma},\sigma)$ at the end of each iteration in Algorithm~\ref{alg_ADMM2} (line 5). If the decrease is sufficiently large, we progress to the next iteration as Algorithm~\ref{alg_ADMM}. Otherwise, we exploit the negative curvature at $\tilde{\mathbf{\sigma}}^k$ by a power method instead.

Due to Lemma 2 in~\cite{Erdogdu21}
, we have
\begin{align}
	\begin{aligned}
		f(\tilde{\sigma}^{k+1}) & \le f(\tilde{\sigma}^k) + t \langle u^k, \text{grad} f(\tilde{\sigma}^k) \rangle + \frac{t^2}{2}\langle u^k, \text{Hess}f(\tilde{\sigma}^k)[u^k]\rangle + \frac{5\|C\|_1}{2}t^3\\
		& \le f(\tilde{\sigma}^k) +\frac{\lambda_{H}(\tilde{\sigma}^k)}{2}t^2 + \frac{5\|C\|_1}{2}t^3,
	\end{aligned} \label{2nd-decrease}
\end{align}
given $\langle u^k, \text{grad} f(\tilde{\sigma}^k) \rangle \le 0$ and $\langle u^k, \text{Hess}f(\tilde{\sigma}^k)[u^k]\rangle = \lambda_{H}(\tilde{\sigma}^k) = \lambda_{\min}(\text{Hess} f(\tilde{\sigma}^k))/2 < 0$.

If we allow an adaptive $t = -\frac{2}{15\|C\|_1}\lambda_{H}(\tilde{\sigma}^k)$, then
\begin{align}
	f(\tilde{\sigma}^k) - f(\tilde{\sigma}^{k+1}) \ge  -\frac{2\lambda_{H}(\tilde{\sigma}^k)^3}{675\|C\|_1^2}, \label{adaptive-t}
\end{align}
where $\lambda_{H}(\tilde{\sigma}^k) \le \frac{1}{2}\lambda_{\min}(\tilde{\sigma}^k) < 0$
The analysis relies on the following result:
\begin{lemma}[Theorem 2 in~\cite{Mei17}]\label{lemma-anyrank-bound}
	For any $\epsilon$-approximate convex point $\sigma \in \mathcal{M}$ of the rank-$r$ non-convex problem~\eqref{sphere-form}, we have
	\begin{align*}
		f(\sigma) \le \text{SDP}(C) - \frac{1}{r-1}(\text{SDP}(C) + \text{SDP}(-C)) + \frac{n}{2}\epsilon,
	\end{align*}
	where $\text{SDP}(C)$ is the optimum of~\eqref{SDP-form}.
\end{lemma}
Suppose we denote $f^* = SDP(C)-1/(k-1)\cdot(SDP(C)+SDP(-C))$, $g(\sigma) = f(\sigma) - f^*$ and assume $g(\tilde{\sigma}^0), \dots, g(\tilde{\sigma}^{T+1}) > 0$. We can then obtain the following, which is the main result of this section:
\begin{theorem}\label{theorem-ADMM-BM-2}
	Suppose we choose an adaptive step $t = \frac{2}{15\|C\|_1}\lambda_{H}(\tilde{\sigma}^k)$. Algorithm~\ref{alg_ADMM2} returns a point $\sigma \in \mathcal{M}$ with
	\begin{align}
		f(\sigma) \le \text{SDP}(C) - \frac{1}{r-1}(\text{SDP}(C) + \text{SDP}(-C)) + \frac{n}{2}\epsilon, \label{obj-2nd-order-bound}
	\end{align}
	within $T = T_1 + T_2$ iterations, where we set $T_1 = \lceil \frac{f(\tilde{\sigma}^0)}{\kappa \epsilon^2} \rceil$ and $T_2 = \lceil 675\|C\|_1^2n/\epsilon^2 \rceil$.
\end{theorem}
\begin{proof}[Proof of Theorem~\ref{theorem-ADMM-BM-2}] \hfill

	Algorithm~\ref{alg_ADMM2} has three possible updates for each iteration: we proceed as Algorithm~\ref{alg_ADMM} when the decrease of $G_{\rho}(\tilde{\sigma}^k,\sigma^k)$ is sufficiently large enough (case 1). Otherwise, we try to exploit the negative curvature along the most negative decreasing direction $u^k$ when the smallest eigenvalue is negative (case 2), or return an $\epsilon-$approximate convex point (case 3).

	\noindent\textbf{Case 1 (line 5-6):} $G_{\rho}(\tilde{\sigma}^{k}, \sigma^{k}) - G_{\rho}(\tilde{\sigma}^{k+1}, \sigma^{k+1}) \ge \Delta$.\\
	We define $\kappa := \left\{\mu - \|C\|^2/\rho, \frac{\rho}{2} \right\}$ and suppose that $\{k_i\}_{i=1,2,\dots,T_1}$ is the set of iteration numbers when
	\begin{align*}
		G_{\rho}(\tilde{\sigma}^{k_i}, \sigma^{k_i}) - G_{\rho}(\tilde{\sigma}^{k_i+1}, \sigma^{k_i+1})
		\ge \kappa \epsilon^2,
	\end{align*}
	is satisfied.

	Summing the decrease for every first-order step, we have
	\begin{align}
		\begin{aligned}
			T_1 \kappa \epsilon^2 & \le \sum_{i=1}^{T_1}\left(G_{\rho}(\tilde{\sigma}^{k_i}, \sigma^{k_i}) - G_{\rho}(\tilde{\sigma}^{k_i+1}, \sigma^{k_i+1}) \right) \le G_{\rho}(\tilde{\sigma}^{k_1}, \sigma^{k_1}) - G_{\rho}(\tilde{\sigma}^{k_{T_1}+1}, \sigma^{k_{T_1}+1}) \\
			& \le G_{\rho}(\tilde{\sigma}^{k_1}, \sigma^{k_1}) - f^* - (G_{\rho}(\tilde{\sigma}^{k_{T_1}+1}, \sigma^{k_{T_1}+1}) - f^*) \le G_{\rho}(\tilde{\sigma}^{k_1}, \sigma^{k_1}) - f^* - g(\tilde{\sigma}^{k_{T_1}+1}) \\
			& \le G_{\rho}(\tilde{\sigma}^{0}, \sigma^{0}) - f^* = g(\tilde{\sigma}^0),
		\end{aligned}
	\end{align}
	i.e. $T_1 \le \frac{f(\tilde{\sigma}^0)}{\kappa \epsilon^2}$.

	\noindent\textbf{Case 2 (line 9-12):} $G_{\rho}(\tilde{\sigma}^{k}, \sigma^{k}) - G_{\rho}(\tilde{\sigma}^{k+1}, \sigma^{k+1}) < \Delta$ and $\lambda_{H}(\tilde{\sigma}^k)  < -\frac{\epsilon}{2}$. \\
	In this case, we have $\lambda_{\min}(\tilde{\sigma}^k) < \lambda_{H}(\tilde{\sigma}^k) < -\frac{\epsilon}{2} < 0$. If $\epsilon$ is set to $-\lambda_{\min}(\tilde{\sigma}^k)$ in Lemma~\ref{lemma-anyrank-bound}, we obtain
	\begin{align*}
		\frac{g(\tilde{\sigma}^k)}{n} \le -\frac{1}{2}\lambda_{\min}(\tilde{\sigma}^k) \le -\lambda_{H}(\tilde{\sigma}^k).
	\end{align*}
	Substituting it into~\eqref{adaptive-t} yields
	\begin{align}
		g(\tilde{\sigma}^k) - g(\tilde{\sigma}^{k+1}) \ge  \frac{2 g(\tilde{\sigma}^k)^3}{675\|C\|_1^2 n^3}.
	\end{align}
	Then
	\begin{align}
		\begin{aligned}
			\frac{1}{g(\tilde{\sigma}^{k+1})^2} - \frac{1}{g(\tilde{\sigma}^{k})^2}
			& \ge \frac{(g(\tilde{\sigma}^k) -g(\tilde{\sigma}^{k+1}))(g(\tilde{\sigma}^k)+g(\tilde{\sigma}^{k+1}))}{g(\tilde{\sigma}^{k})^2g(\tilde{\sigma}^{k+1})^2} \\
			& \ge \frac{2}{675\|C\|_1^2 n^3}\left( \frac{g(\tilde{\sigma}^k)}{g(\tilde{\sigma}^{k+1})} + \frac{g(\tilde{\sigma}^k)^2}{g(\tilde{\sigma}^{k+1})^2}\right) \\
			& \ge \frac{4}{675\|C\|_1^2 n^3}. \label{adaptive-t-inverse2}
		\end{aligned}
	\end{align}

	\noindent\textbf{Case 3 (line 14):} $G_{\rho}(\tilde{\sigma}^{k}, \sigma^{k}) - G_{\rho}(\tilde{\sigma}^{k+1}, \sigma^{k+1}) < \Delta$ and $\lambda_{H}(\tilde{\sigma}^k)  > -\frac{\epsilon}{2}$. \\
	We can obtain
	\begin{align*}
		\lambda_{\min}(\text{Hess} f(\tilde{\sigma}^k)) \ge 2*\lambda_{H}(\tilde{\sigma}^k)  > -\epsilon,
	\end{align*}
	which means $\tilde{\sigma}^k$ is already an $\epsilon$-approximate convex point and \eqref{obj-2nd-order-bound} is satisfied due to Lemma~\ref{lemma-gamma-bound}.

	Suppose $\{k_j\}_{j=1,2,\dots,T_2}$ is the subsequence of $\{1,\dots,T +1 \}$ such that the negative curvature is taken but not at an $\epsilon$-approximate convex point. Note that $f(\tilde{\sigma}^{k_j+1}) = G_{\rho}(\tilde{\sigma}^{k_j+1}, \sigma^{k_j+1}), \forall j$ and $G_{\rho}(\tilde{\sigma}^{k_j+1}, \sigma^{k_j+1}) \ge G_{\rho}(\tilde{\sigma}^{k_{j+1}}, \sigma^{k_{j+1}}) \ge f(\tilde{\sigma}^{k_{j+1}})$ due to the non-decreasing property of $G_{\rho}(\tilde{\sigma}^{k}, \sigma^{k})$, we can obtain $g(\tilde{\sigma}^{k_{j}+1}) \ge g(\tilde{\sigma}^{k_{j+1}})$. Therefore, summing up~\eqref{adaptive-t-inverse2} yields
	\begin{align}
		\begin{aligned}
			\frac{4T_2}{675\|C\|_1^2 n^3} \le \sum_{j=1}^{T_2} \left(\frac{1}{g(\tilde{\sigma}^{k_j+1})^2} - \frac{1}{g(\tilde{\sigma}^{k_j})^2} \right) \le \frac{1}{g(\tilde{\sigma}^{k_T+1})^2} - \frac{1}{g(\tilde{\sigma}^{k_1})^2}  \le \frac{1}{g(\tilde{\sigma}^{T+1})^2},
		\end{aligned}
	\end{align}
	which guarantees $g(\tilde{\sigma}^{T+1}) \le n\epsilon/2$ as long as we set $T_2 = \lceil 675\|C\|_1^2n/\epsilon^2 \rceil$.
\end{proof}

\section{Extension to the product of Stiefel manifolds}\label{section-proximal-variant}
If we replace diagonal constraints with block-diagonal constraints in problem in~\eqref{SDP-form}, we obtain
\begin{align}
	\begin{aligned}
		\min \quad & \langle C, X \rangle \\
		\text{s.t} \quad & X_{ii} = I_d, \text{for } i \in [q] \\
		& \quad X \in \mathbb{S}_+^n,
	\end{aligned} \label{SDP-extended-form}
\end{align}
with $n = qd$ and the rank-constrained counterpart is
\begin{align}
	\begin{aligned}
		\min \quad &  f(\sigma): = \langle C, \sigma \sigma^\top \rangle \\
		\text{s.t.} \quad & \sigma_i^\top \sigma_i = I_d, \text{for } i \in [q],
	\end{aligned} \label{stiefel-form}
\end{align}
where $\mathbf{\sigma} := [\sigma_1, \sigma_2, ..., \sigma_q]^\top \in \mathbb{R}^{qd \times r}$ with $r \ge d$ and $\sigma_i \in \mathbb{R}^{r \times d}$ is the $i$-th block of $\sigma$, and the manifold $\mathcal{M}$ is generalized to the product of Stiefel manifolds $\mathcal{M} = \{[\sigma_1, \sigma_2, \dots, \sigma_q]^\top \in \mathbb{R}^{nd \times r} \mid \ \sigma_i^\top \sigma_i = I_d, \ i \in [q]\} := \mathcal{M}_1 \times \cdots \times \mathcal{M}_q$. To ensure the convergence of Algorithm~\ref{alg_ADMM}, we add a proximal regularization to the update of $\tilde{\sigma}$. When we add a proximal regularization to step~\eqref{ADMM-1}, we obtain a new prox-ADMM algorithm as follows,
\begin{subequations}
	\begin{align}
		\tilde{\sigma}^{k+1} = & \quad \mathop{\text{argmin}}_{\tilde{\sigma} \in \mathcal{M}} \  L_{\rho}(\tilde{\sigma}, \sigma^k, y^k) + \frac{\mu}{2}\|\tilde{\sigma} - \tilde{\sigma}^k\|^2, \label{prox-tilde-sigma}\\
		\sigma^{k+1} = &\quad \mathop{\text{argmin}}_{\sigma} \ L_{\rho}(\tilde{\sigma}^{k+1}, \sigma, y^k), \\
		y^{k+1} = & \quad y^k + \rho(\tilde{\sigma}^{k+1} - \sigma^{k+1}).
	\end{align}
\end{subequations}
The update of $\tilde{\sigma}$ can be written as $\tilde{\sigma}^{k+1} \gets \text{Proj}_{\mathcal{M}} (\gamma^k )$ if we define
\begin{align*}
	\gamma^k := \frac{\mu}{\rho + \mu}\tilde{\sigma}^k + \frac{\rho}{\rho + \mu}\sigma^k - \frac{1}{\rho + \mu} (y^k + C \sigma^k),
\end{align*}
which is equivalent to
\begin{align*}
	\tilde{\sigma}_i^{k+1} \gets \text{Proj}_{\mathcal{M}_i} (\gamma_i^k), \ \forall i \in [n],
\end{align*}
with $\gamma^k = [\gamma_1^k, \dots, \gamma_n^k]^\top, \gamma_i^k \in \mathbb{R}^{r \times d}$.  The projection has an analytic solution by computing the SVD factorization of each $\gamma_i^k$, which we show in the following result:
\begin{lemma}[Theorem 1~\cite{Lai14}]\label{lemma-project-Stiefel}
	The constrained quadratic problem
	\begin{align*}
		P^* = \mathop{\text{argmin}}_{P \in \mathbb{R}^{r \times d}} \frac{1}{2}\|P - X\|_F^2, \quad s.t. \ P^\top P = I_d,
	\end{align*}
	which is the projection of $X$ to the Stiefel manifold $P^\top P = I_d$, has closed-form solution $P^* = UI_{r \times d}V^\top$, where $U \in \mathbb{R}^{r \times r}, V \in \mathbb{R}^{d \times d}$ are two orthogonal matrices and $D \in \mathbb{R}^{r \times d}$ is a diagonal matrix satisfying the SVD factorization $X = UDV^\top$.
\end{lemma}
Hence, we have a proximal ADMM algorithm for the products of Stiefel manifolds shown in Algorithm~\ref{alg_proxADMM}. Compared with Algorithm~\ref{alg_ADMM}, the only difference is an additional regularization for the update of $\tilde{\sigma}^{k+1}$ (line 2) to guarantee the convergence of ADMM-BM theoretically.
\begin{algorithm}[H]
	\caption{Proximal ADMM Burer-Monteiro algorithm (Prox-ADMM-BM)}
	\label{alg_proxADMM}
	\begin{algorithmic}[1]
		\State
		Initialization: set $\sigma^0 = \tilde{\sigma}^0 \in \mathcal{M},y^0 = C \tilde{\sigma}^0$.
		\\[7pt]
		\While {\textit{termination criteria not satisfied}}
		\State $\gamma^k \gets \frac{\mu}{\rho + \mu}\tilde{\sigma}^k + \frac{\rho}{\rho + \mu}\sigma^k - \frac{1}{\rho + \mu} (y^k + C \sigma^k)$
		\State $\gamma_i^k \gets \text{Proj}_{\mathcal{M}_i}(\gamma_i^k), \ \forall i \in [q]$
		\State $\sigma^{k+1} \gets \tilde{\sigma}^{k+1} + \frac{1}{\rho}(y^k - C\tilde{\sigma}^{k+1})$
		\State $y^{k+1} \gets y^k + \rho(\tilde{\sigma}^{k+1} - \sigma^{k+1})$
		\EndWhile
	\end{algorithmic}
\end{algorithm}
Similar to Theorem~\ref{theorem-global-1st}, we can prove that Algorithm~\ref{alg_proxADMM} converges to a critical point of $f$.
\begin{theorem} \label{theorem-global-1st-prox}
	Suppose $\forall i \in [q]$, $\|\gamma_i^k\|_F$ is nonzero for any iteration $k$. If we set $\rho, \mu \ge 0$ properly such that $\mu - {\|C\|^2}/{\rho} > 0$, then the sequence $(\tilde{\sigma}^k)_{k \in \mathbb{N}}$ generated by Algorithm~\ref{alg_proxADMM} satisfies Lemma~\ref{lemma-KL-convergence} and converges to a critical point $\bar{\sigma}$ of $f$. Moreover, the sequence $(\tilde{\sigma}^k)_{k \in \mathbb{N}}$ has a finite length, i.e. $\sum_{k=0}^{+\infty} \|\tilde{\sigma}^{k+1} - \tilde{\sigma}^k\| < \infty$.
\end{theorem}

\begin{proof}[Proof of Theorem~\ref{theorem-global-1st-prox}]
	Similar to~\eqref{decrease-1}, we have the decrease
	\begin{align*}
		L_{\rho}(\tilde{\sigma}^k, \sigma^k, y^k) \ge L_{\rho}(\tilde{\sigma}^{k+1}, \sigma^k, y^k) + \frac{\mu}{2} \|\tilde{\sigma}^k - \tilde{\sigma}^{k+1}\|_F^2,
	\end{align*}
	and the counterpart of Lemma~\ref{lemma-monotonic-decrease} is
	\begin{align*}
		L_{\rho}(\tilde{\sigma}^k, \sigma^k, y^k) - L_{\rho}(\tilde{\sigma}^{k+1}, \sigma^{k+1}, y^{k+1})
		\ge \left( \mu - \frac{\|C\|^2}{\rho} \right) \|\tilde{\sigma}^{k+1} - \tilde{\sigma}^{k}\|^2_F + \frac{\rho}{2}\|\sigma^{k+1} - \sigma^k\|^2_F,
	\end{align*}
	which implies (C1) is valid for $G_{\rho}(\tilde{\sigma}, \sigma)$ since
	\begin{align}
		\begin{aligned}
			G_{\rho}(\tilde{\sigma}^k, \sigma^k) - G_{\rho}(\tilde{\sigma}^{k+1}, \sigma^{k+1})
			\ge & \left( \mu - \frac{\|C\|^2}{\rho} \right) \cdot \|\tilde{\sigma}^{k+1} - \tilde{\sigma}^{k} \|^2_F + \frac{\rho}{2} \|\sigma^{k+1} - \sigma^k \|^2_F,\\
			\ge & \min \left\{\mu - \frac{\|C\|^2}{\rho}, \frac{\rho}{2} \right\} \cdot \bigg{\|}\begin{matrix}
				\tilde{\sigma}^{k+1} - \tilde{\sigma}^k\\
				\sigma^{k+1} - \sigma^k
			\end{matrix}\bigg{\|}_F^2.
		\end{aligned} \label{prox-decrease}
	\end{align}
	In addition, the condition~\eqref{1st-tilde-sigma} becomes
	\begin{align*}
		C\sigma^k + y^k + \rho(\tilde{\sigma}^{k+1} - \sigma^k) + \mu(\tilde{\sigma}^{k+1} - \tilde{\sigma}^{k})+ v^{k+1} = 0.
	\end{align*}
	Then a subgradient of $G_{\rho}(\tilde{\sigma}, \sigma)$ at $(\tilde{\sigma}^{k+1}, \sigma^{k+1})$ is
	\begin{align*}
		& \partial_{\tilde{\sigma}}G_{\rho}(\tilde{\sigma}^{k+1}, \sigma^{k+1}) \\
		= \ & 2C\tilde{\sigma}^{k+1} + \rho(\tilde{\sigma}^{k+1} - \sigma^{k+1}) + v^{k+1}, \notag\\
		= \ & 2C\tilde{\sigma}^{k+1} + \rho(\tilde{\sigma}^{k+1} - \sigma^{k+1}) - [C\sigma^k + y^k + \rho(\tilde{\sigma}^{k+1} - \sigma^k) + \mu(\tilde{\sigma}^{k+1} - \tilde{\sigma}^{k})], \notag\\
		= \ & 2C\tilde{\sigma}^{k+1} - C\sigma^k - C\tilde{\sigma}^k - \rho(\sigma^{k+1} - \sigma^{k}) - \mu(\tilde{\sigma}^{k+1} - \tilde{\sigma}^{k}) \notag\\
		= \ & C(\tilde{\sigma}^{k+1} - \tilde{\sigma}^k) + C(\sigma^{k+1} - \sigma^k) +C(\tilde{\sigma}^{k+1} - \sigma^{k+1}) - \rho(\sigma^{k+1} - \sigma^{k}) - \mu(\tilde{\sigma}^{k+1} - \tilde{\sigma}^{k})\notag\\
		\overset{\eqref{ADMM-3}}{=} & C(\tilde{\sigma}^{k+1} - \tilde{\sigma}^k) + C(\sigma^{k+1} - \sigma^k) +\frac{C}{\rho}(y^{k+1}-y^k) - \rho(\sigma^{k+1} - \sigma^{k}) - \mu(\tilde{\sigma}^{k+1} - \tilde{\sigma}^{k}) \notag\\
		\overset{\eqref{prim-dual-link}}{=} & (C - \mu I) (\tilde{\sigma}^{k+1} - \tilde{\sigma}^k) + C(\sigma^{k+1} - \sigma^k) +\frac{C^2}{\rho}(\tilde{\sigma}^{k+1}-\tilde{\sigma}^k) - \rho(\sigma^{k+1} - \sigma^{k}),
	\end{align*}
	and
	\begin{align*}
		\partial_{\sigma}G_{\rho}(\tilde{\sigma}^{k+1}, \sigma^{k+1}) = & \rho(\sigma^{k+1} - \tilde{\sigma}^{k+1}) \overset{\eqref{ADMM-3}}{=} y^k - y^{k+1} \overset{\eqref{prim-dual-link}}{=} C(\tilde{\sigma}^k - \tilde{\sigma}^{k+1}),
	\end{align*}
	which imply that
	\begin{align*}
		\bigg{\|}\begin{matrix}
			\partial_{\tilde{\sigma}}G_{\rho}(\tilde{\sigma}^{k+1}, \sigma^{k+1})\\
			\partial_{\sigma}G_{\rho}(\tilde{\sigma}^{k+1}, \sigma^{k+1})
		\end{matrix}\bigg{\|}_F
		\le \left( 2\|C\|+ \max \{ \rho,\mu \} +\frac{\|C\|^2}{\rho} \right)
		\bigg{\|}\begin{matrix}
			\tilde{\sigma}^{k+1} - \tilde{\sigma}^k\\
			\sigma^{k+1} - \sigma^k
		\end{matrix}\bigg{\|}_F, \ \forall k,
	\end{align*}
	and hence (C2) is satisfied. Moreover, Lemma~\ref{lemma-lower-obj} shows that $L_{\rho}(\tilde{\sigma}^k, \sigma^k, y^k) = G_{\rho}(\tilde{\sigma}^k, \sigma^k)$ is uniformly lower-bounded. In combination with \eqref{prox-decrease}, we can prove that $G_{\rho}(\tilde{\sigma}^k, \sigma^k)$ converges to a limiting point, which implies
	\begin{align*}
		\bigg{\|}\begin{matrix}
			\tilde{\sigma}^{k+1} - \tilde{\sigma}^k\\
			\sigma^{k+1} - \sigma^k
		\end{matrix}\bigg{\|}_F \to 0,
	\end{align*}
	and then
	\begin{align*}
		\|\tilde{\sigma}^{k+1} - \sigma^{k+1}\| = \frac{1}{\rho}\|y^{k+1} - y^k\| \overset{\eqref{prim-dual-link}}{\le}\|C\| \cdot \|\tilde{\sigma}^{k+1} - \tilde{\sigma}^k\| \to 0.
	\end{align*}
	Following the same argument in Theorem~\ref{theorem-global}, we can show that the distance between $\{\tilde{\sigma}^k\}$ and the set of first order stationary point is $0$, and that the compactness of $\mathcal{M}$ implies that there is a subsequence of $\tilde{\sigma}^k$ that converges to a first-order stationary point, i.e. (C3) is valid. We hence prove Theorem~\ref{theorem-global-1st-prox} via Lemma~\ref{lemma-KL-convergence}.
\end{proof}

\section{Experimental results}\label{section-experiments}
In this section, we describe our computational results for Algorithm \ref{alg_ADMM} and Algorithm~\ref{alg_proxADMM}, named ADMM-BM and Prox-ADMM-BM respectively. All experiments are run on a PC with 3.0 GHz processor and 16GB memory, with all code written in Julia. For comparison, we also present computational results from Riemannian gradient (RGD) method and Riemannian trust-region (RTR) method from Manopt.jl~\cite{Bergmann22}. The benchmarking metric is the relative optimality gap defined as:
\begin{align}
	\text{relative optimality gap = } \left\vert \frac{\langle \tilde{\sigma}^k, C \tilde{\sigma}^k \rangle - \langle {\sigma^*}, C \sigma^* \rangle}{\langle {\sigma^*}, C \sigma^* \rangle} \right\vert, \label{relative-optimality}
\end{align}
where $\tilde{\sigma}^k$ is generated at the $k$-th iteration by Algorithm~\ref{alg_ADMM} or Algorithm~\ref{alg_proxADMM} that is feasible for~\eqref{sphere-form} or~\eqref{stiefel-form}, and $\sigma^*$ is the optimal solution obtained via Mosek \cite{mosek}.

\subsection{Max-cut}
For Algorithm~\ref{alg_ADMM}, we set $r = \lceil {\sqrt{2n}} \rceil$ and $\rho$ to $\|C\|$ and test it on the dataset Gset\footnote{\url{https://www.cise.ufl.edu/research/sparse/matrices/Gset/}}, which contains a collection of max-cut problems of the form~\eqref{SDP-form}. An initial condition $\tilde{\sigma}^0  \in \mathbb{R}^{n \times r}$ is generated uniformly at random on $[0,1]$ and then normalized row-wise. We also set $\sigma^0 = \tilde{\sigma}^0$ and $y^0 = C \tilde{\sigma}^0$ for ADMM-BM and $\tilde{\sigma}^0$ as the initial point for RGD and RTR. We test problems with sizes varying from $n=800$ to $n=10000$ and the results are shown in Figure~\ref{fig:test-maxcut}, where the x-axis denotes the computation time and the y-axis denotes the relative optimality defined in~\eqref{relative-optimality}. ADMM-BM always performs better than RGD and RTR at moderate accuracy $10^{-4}$. The computation time is within several seconds even for problem sizes up to $n=10000$.

\begin{figure}[ht]
	\centering
	\begin{tabular}{c@{}c@{}c@{}c@{}}
		&   \scriptsize{G1, $n=800$, $\|C\|=12.197$} & \scriptsize{G34, $n=2000$, $\|C\| = 0.899$}
		\vspace{-1mm}
		\\
		\rotatebox[origin=c]{90}{\scriptsize{Relative optimality}}  &
		\includegraphics[width=0.45\linewidth]{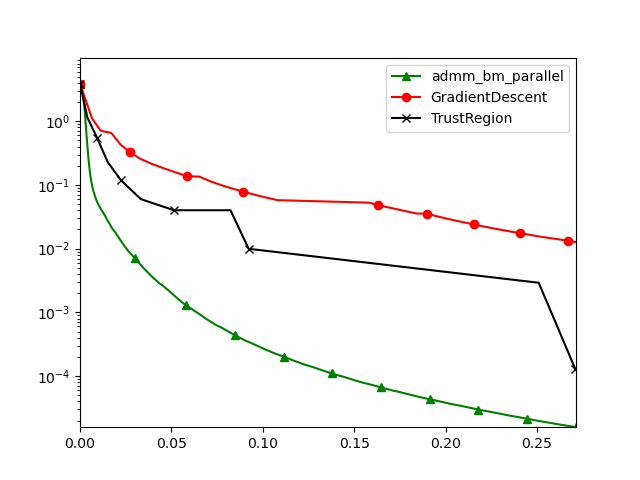}  &  \includegraphics[width=0.45\linewidth]{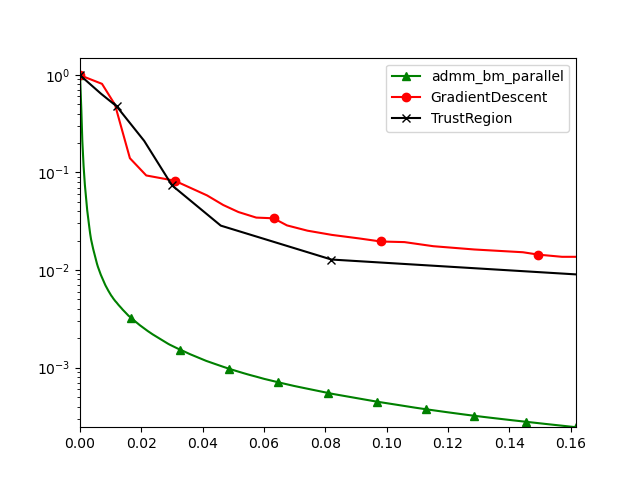}
		\vspace{-11mm}
		\\
		&   \scriptsize{time: $s$} & \scriptsize{time: $s$}
		\\
		&   \scriptsize{G57, $n=5000$, $\|C\| = 0.889$} & \scriptsize{G67, $n=10000$, $\|C\|=1.676$}
		\vspace{-1mm}
		\\
		&
		\includegraphics[width=0.45\linewidth]{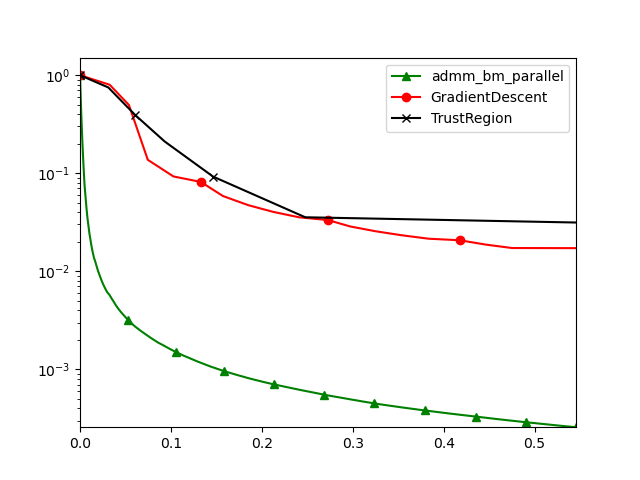}& \includegraphics[width=0.45\linewidth]{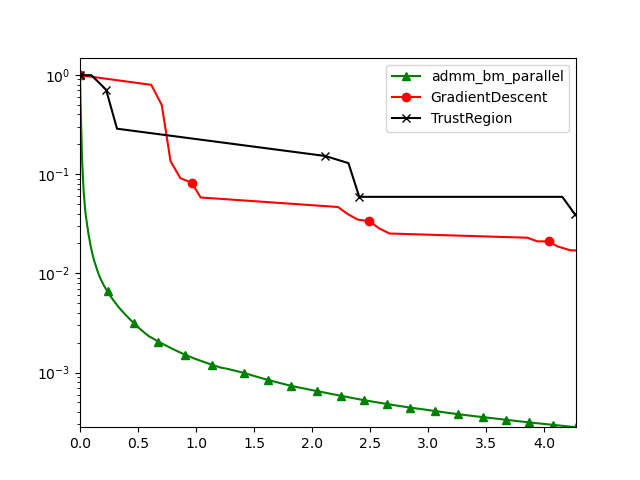}
		\vspace{-4mm}
		\\
		&   \scriptsize{time: $s$} & \scriptsize{time: $s$} \\
	\end{tabular}
	\caption{Tests of Algorithm~\ref{alg_ADMM} on max-cut problems.}
	\label{fig:test-maxcut}
\end{figure}

\subsection{SO(3) Synchronization}
The SO(3) synchronization problem~\cite{Mei17} is a typical problem of the form~\eqref{SDP-extended-form} where we choose $d=3$. In order to test Algorithm~\ref{alg_proxADMM}, each entry of $\tilde{\sigma}^0$ is generated uniformly at random on $[0,1]$ and then projected back to the product of Stiefel manifolds as in Lemma~\ref{lemma-project-Stiefel}. We initialize $\sigma^0 = \tilde{\sigma}^0, y^0 = C \tilde{\sigma}^0$ in Prox-ADMM-BM and $\tilde{\sigma}^0$ as the start point for RTG and RTR. In addition, penalty parameters are set to $\rho = \mu = \|C\|_2$ in Prox-ADMM-BM. We change the dimension from $n=300$ to $n=15000$ and also choose two sparsity pattern, $0.02$ and $0.002$. As shown in Figure \ref{fig:test-stiefel}, Prox-ADMM-BM always performs better than RTG and RTR up to accuracy level $10^{-4}$, while RTR performs better on a higher accuracy level for some cases due to local super-linear convergence. In addition, we find that the speed of Prox-ADMM-BM converges faster when the cost matrix $C$ is more sparse, which is reasonable since most of computational time is spent on matrix products when we update $\tilde{\sigma}$ and $\sigma$. This also applies to ADMM-BM.

\begin{figure}[ht]
	\centering
	\begin{tabular}{c@{}c@{}c@{}c@{}}
		&   \scriptsize{ $n=300, s=0.02$, $\|C\|_2=3.685$} & \scriptsize{ $n=3000, s=0.02$, $\|C\|_2=30.625$}
		\vspace{-1mm}
		\\
		\rotatebox[origin=c]{90}{\scriptsize{Relative optimality}}  &
		\includegraphics[width=0.475\linewidth]{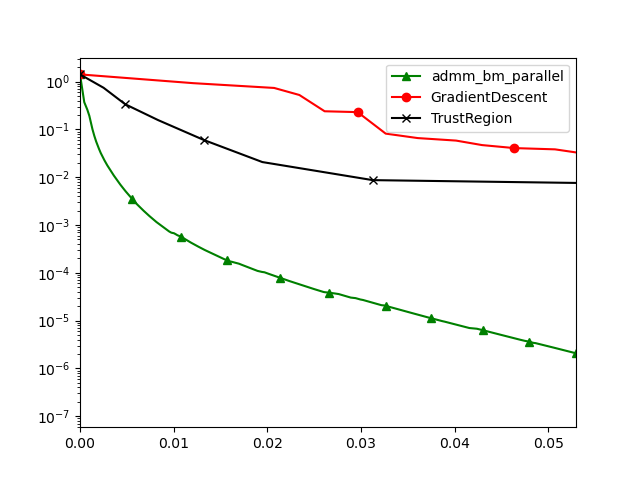}  &  \includegraphics[width=0.475\linewidth]{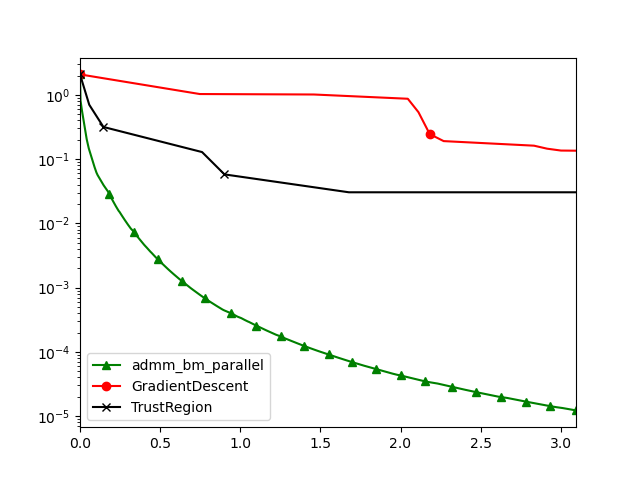}
		\vspace{-11mm}
		\\
		&   \scriptsize{time: $s$} & \scriptsize{time: $s$}
		\\
		&   \scriptsize{$n=9000, s= 0.02$, $\|C\|_2= 90.588$} & \scriptsize{$n=9000, s= 0.002$, $\|C\|_2= 9.712$}
		\vspace{-1mm}
		\\
		&
		\includegraphics[width=0.475\linewidth]{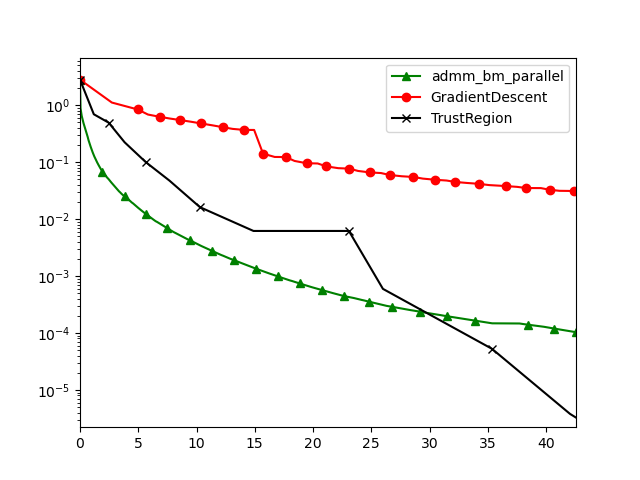}& \includegraphics[width=0.475\linewidth]{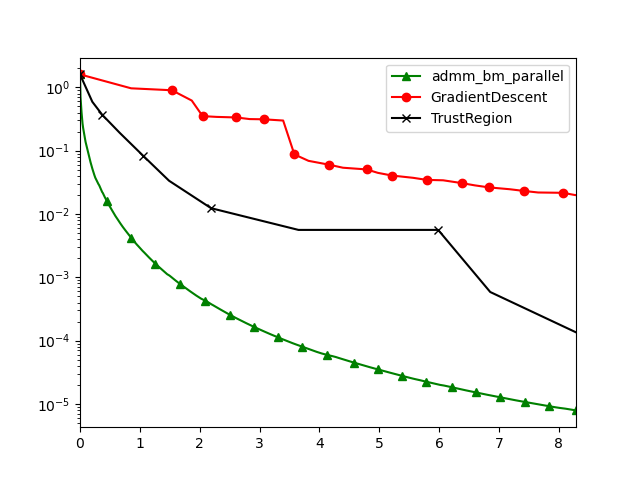}
		\vspace{-4mm}
		\\
		&   \scriptsize{time: $s$} & \scriptsize{time: $s$} \\
		&   \scriptsize{ $n=15000, s= 0.02$, $\|C\|_2= 150.843$} & \scriptsize{$n=15000, s= 0.002$, $\|C\|_2= 15.762$}
		\vspace{-1mm}
		\\
		&
		\includegraphics[width=0.475\linewidth]{Stiefel_n=3000_d=3_s=2e-2.png}& \includegraphics[width=0.475\linewidth]{Stiefel_n=3000_d=3_s=2e-3.png}
		\vspace{-4mm}
		\\
	\end{tabular}
	\caption{Tests of Algorithm~\ref{alg_proxADMM} on SO(3) synchronization problems.}
	\label{fig:test-stiefel}
\end{figure}

\section{Contribution \& Future Work}
In this paper, we proposed an ADMM-based Burer Monteiro method for diagonally constrained SDPs. It introduced a bilinear decomposition to the original problem such that all updates can be run in parallel. Theoretically, we proved a global convergence of Algorithm~\ref{alg_ADMM} to a first-order stationary point and showed that the algorithm scaled well to large-scale diagonal SDPs and converged to a global minimum empirically. We also introduced Algorithm~\ref{alg_ADMM2} with negative curvature exploitation, which guarantees a solution to $1 - O(1/r)$ approximation to SDPs. Moreover, Algorithm~\ref{alg_ADMM} can be generalized to Algorithm~\ref{alg_proxADMM} that solves SDPs with block-diagonal constraints with a global convergence to a first-order stationary point. Our experiments show that the proposed algorithm and its proximal variant outperform Riemannian manifold algorithms at moderate accuracy and both are suitable for large-scale SDPs.

\newpage
\bibliography{reference}

\newcommand{\etalchar}[1]{$^{#1}$}
\begin{thebibliography}{BEGFB94}

\bibitem[ABS13]{Attouch13}
Hedy Attouch, Jérôme Bolte, and Benar~Fux Svaiter.
\newblock Convergence of descent methods for semi-algebraic and tame problems:
  proximal algorithms, forward–backward splitting, and regularized
  gauss–seidel methods.
\newblock {\em Mathematical Programming}, 137(1):91--129, 2013.

\bibitem[ADV10]{Andersen10}
Martin~S. Andersen, Joachim Dahl, and Lieven Vandenberghe.
\newblock Implementation of nonsymmetric interior-point methods for linear
  optimization over sparse matrix cones.
\newblock {\em Mathematical Programming Computation}, 2(3):167--201, 2010.

\bibitem[AHO97]{Alizadeh97}
Farid Alizadeh, Jean-Pierre~A. Haeberly, and Michael~L. Overton.
\newblock Complementarity and nondegeneracy in semidefinite programming.
\newblock {\em Mathematical Programming}, 77(1):111--128, 1997.

\bibitem[AM19]{Ahmadi19}
Amir~Ali Ahmadi and Anirudha Majumdar.
\newblock {DSOS} and {SDSOS} optimization: More tractable alternatives to sum
  of squares and semidefinite optimization.
\newblock {\em SIAM Journal on Applied Algebra and Geometry}, 3(2):193--230,
  2019.

\bibitem[AMS07]{Absil07}
P.-A. Absil, R.~Mahony, and R.~Sepulchre.
\newblock {\em Optimization Algorithms on Matrix Manifolds}.
\newblock Princeton University Press, USA, 2007.

\bibitem[ApS21]{mosek}
MOSEK ApS.
\newblock {\em MOSEK Modeling Cookbook}, 2021.

\bibitem[BBV16]{Bandeira16}
Afonso~S. Bandeira, Nicolas Boumal, and Vladislav Voroninski.
\newblock On the low-rank approach for semidefinite programs arising in
  synchronization and community detection.
\newblock In {\em Proceedings of the 29th Conference on Learning Theory,
  {COLT}}, volume~49 of {\em {JMLR} Workshop and Conference Proceedings}, pages
  361--382, 2016.

\bibitem[BC11]{Bauschke11}
Heinz~H. Bauschke and Patrick~L. Combettes.
\newblock {\em Convex Analysis and Monotone Operator Theory in Hilbert Spaces}.
\newblock 1st edition, 2011.

\bibitem[BEGFB94]{Boyd94}
Stephen Boyd, Laurent El~Ghaoui, Eric Feron, and Venkataramanan Balakrishnan.
\newblock {\em Linear Matrix Inequalities in System and Control Theory}.
\newblock Society for Industrial and Applied Mathematics, Philadelphia, 1994.

\bibitem[Ber22]{Bergmann22}
Ronny Bergmann.
\newblock Manopt.jl: Optimization on manifolds in {J}ulia.
\newblock {\em Journal of Open Source Software}, 7(70):3866, 2022.

\bibitem[BGSB19]{Banjac19}
G.~Banjac, P.~Goulart, B.~Stellato, and S.~Boyd.
\newblock Infeasibility detection in the alternating direction method of
  multipliers for convex optimization.
\newblock {\em Journal of Optimization Theory and Applications},
  183(2):490--519, 2019.

\bibitem[Bou23]{Boumal20}
Nicolas Boumal.
\newblock {\em An introduction to optimization on smooth manifolds}.
\newblock Cambridge University Press, Cambridge, Jun 2023.

\bibitem[BPC{\etalchar{+}}11]{Boyd11}
Stephen Boyd, Neal Parikh, Eric Chu, Borja Peleato, and Jonathan Eckstein.
\newblock Distributed optimization and statistical learning via the alternating
  direction method of multipliers.
\newblock {\em Foundations and Trends\textsuperscript{\textregistered} in
  Machine Learning}, 3(1):1--122, 2011.

\bibitem[BS21]{Barber21}
Rina~Foygel Barber and Emil~Y. Sidky.
\newblock Convergence for nonconvex admm, with applications to {CT} imaging.
\newblock {\em arXiv preprint arXiv:2006.07278}, 2021.

\bibitem[BST14]{Bolte14}
Jérôme Bolte, Shoham Sabach, and Marc Teboulle.
\newblock Proximal alternating linearized minimization for nonconvex and
  nonsmooth problems.
\newblock {\em Mathematical Programming}, 146(1):459--494, 2014.

\bibitem[BV97]{Boyd97}
Stephen Boyd and Lieven Vandenberghe.
\newblock {\em Semidefinite Programming Relaxations of Non-Convex Problems in
  Control and Combinatorial Optimization}, pages 279--287.
\newblock Springer US, Boston, MA, 1997.

\bibitem[BVB16]{Boumal16}
Nicolas Boumal, Vlad Voroninski, and Afonso Bandeira.
\newblock The non-convex {B}urer-{M}onteiro approach works on smooth
  semidefinite programs.
\newblock In D.~Lee, M.~Sugiyama, U.~Luxburg, I.~Guyon, and R.~Garnett,
  editors, {\em Advances in Neural Information Processing Systems}, volume~29,
  page 2765–2773, 2016.

\bibitem[BYZ00]{Benson00}
Steven~J. Benson, Yinyu Ye, and Xiong Zhang.
\newblock Solving large-scale sparse semidefinite programs for combinatorial
  optimization.
\newblock {\em SIAM Journal on Optimization}, 10(2):443--461, 2000.

\bibitem[CG22]{Chen22}
Yuwen Chen and Paul Goulart.
\newblock Burer-{M}onteiro {ADMM} for large-scale diagonally constrained
  {SDP}s.
\newblock In {\em 2022 European Control Conference (ECC)}, pages 66--71, 2022.

\bibitem[DHNY20]{Duchi20}
John~C Duchi, Oliver Hinder, Andrew Naber, and Yinyu Ye.
\newblock Conic descent and its application to memory-efficient optimization
  over positive semidefinite matrices.
\newblock In H.~Larochelle, M.~Ranzato, R.~Hadsell, M.~F. Balcan, and H.~Lin,
  editors, {\em Advances in Neural Information Processing Systems}, volume~33,
  pages 8308--8317, 2020.

\bibitem[DYC{\etalchar{+}}21]{Ding20}
Lijun Ding, Alp Yurtsever, Volkan Cevher, Joel~A. Tropp, and Madeleine Udell.
\newblock An optimal-storage approach to semidefinite programming using
  approximate complementarity.
\newblock {\em SIAM Journal on Optimization}, 31(4):2695--2725, 2021.

\bibitem[EB92]{Eckstein92}
Jonathan Eckstein and Dimitri~P. Bertsekas.
\newblock On the {D}ouglas--{R}achford splitting method and the proximal point
  algorithm for maximal monotone operators.
\newblock {\em Mathematical Programming}, 55(1):293--318, 1992.

\bibitem[EDM17]{Erdogdu17}
Murat~A Erdogdu, Yash Deshpande, and Andrea Montanari.
\newblock Inference in graphical models via semidefinite programming
  hierarchies.
\newblock In {\em Advances in Neural Information Processing Systems},
  volume~30, 2017.

\bibitem[EOPV21]{Erdogdu21}
Murat~A. Erdogdu, Asuman Ozdaglar, Pablo~A. Parrilo, and Nuri~Denizcan Vanli.
\newblock Convergence rate of block-coordinate maximization burer-monteiro
  method for solving large {SDPs}.
\newblock {\em Mathematical Programming}, 2021.

\bibitem[FGM17]{Freund15}
Robert~M. Freund, Paul Grigas, and Rahul Mazumder.
\newblock An extended {Frank-Wolfe} method with "in-face" directions, and its
  application to low-rank matrix completion.
\newblock {\em SIAM Journal on Optimization}, 27(1):319--346, 2017.

\bibitem[FKMN00]{Fukuda00}
Mituhiro Fukuda, Masakazu Kojima, Kazuo Murota, and Kazuhide Nakata.
\newblock Exploiting sparsity in semidefinite programming via matrix completion
  i: General framework.
\newblock {\em SIAM Journal on Optimization}, 11(3):647–674, March 2000.

\bibitem[GCG21]{Garstka21}
Michael Garstka, Mark Cannon, and Paul Goulart.
\newblock Cosmo: A conic operator splitting method for convex conic problems.
\newblock {\em Journal of Optimization Theory and Applications},
  190(3):779--810, 2021.

\bibitem[GHV20]{Gepp20}
Adrian Gepp, Geoff Harris, and Bruce Vanstone.
\newblock {Financial applications of semidefinite programming: a review and
  call for interdisciplinary research}.
\newblock {\em Accounting and Finance}, 60(4):3527--3555, December 2020.

\bibitem[GW95]{Goemans95}
Michel~X. Goemans and David~P. Williamson.
\newblock Improved approximation algorithms for maximum cut and satisfiability
  problems using semidefinite programming.
\newblock {\em Journal of the ACM}, 42(6):1115–1145, 1995.

\bibitem[HLR16]{Hong16}
Mingyi Hong, Zhi-Quan Luo, and Meisam Razaviyayn.
\newblock Convergence analysis of alternating direction method of multipliers
  for a family of nonconvex problems.
\newblock {\em SIAM Journal on Optimization}, 26(1):337--364, 2016.

\bibitem[HP10]{Huang10}
Yongwei Huang and Daniel~P. Palomar.
\newblock Rank-constrained separable semidefinite programming with applications
  to optimal beamforming.
\newblock {\em IEEE Transactions on Signal Processing}, 58(2):664--678, 2010.

\bibitem[HRVW96]{Helmberg96}
Christoph Helmberg, Franz Rendl, Robert~J. Vanderbei, and Henry Wolkowicz.
\newblock An interior-point method for semidefinite programming.
\newblock {\em SIAM Journal on Optimization}, 6(2):342--361, 1996.

\bibitem[Jag13]{Jaggi13}
Martin Jaggi.
\newblock Revisiting {{Frank}-{Wolfe}}: Projection-free sparse convex
  optimization.
\newblock In {\em Proceedings of the 30th International Conference on Machine
  Learning}, volume~28 of {\em Proceedings of Machine Learning Research}, pages
  427--435, 2013.

\bibitem[JLMZ19]{Jiang19}
Bo~Jiang, Tianyi Lin, Shiqian Ma, and Shuzhong Zhang.
\newblock Structured nonconvex and nonsmooth optimization: algorithms and
  iteration complexity analysis.
\newblock {\em Computational Optimization and Applications}, 72(1):115--157,
  2019.

\bibitem[KGB16]{Kovnatsky16}
Artiom Kovnatsky, Klaus Glashoff, and Michael~M. Bronstein.
\newblock {MADMM:} {A} generic algorithm for non-smooth optimization on
  manifolds.
\newblock volume 9909 of {\em Lecture Notes in Computer Science}, pages
  680--696, 2016.

\bibitem[LHW17]{Lu17}
Songtao Lu, Mingyi Hong, and Zhengdao Wang.
\newblock A nonconvex splitting method for symmetric nonnegative matrix
  factorization: Convergence analysis and optimality.
\newblock {\em IEEE Transactions on Signal Processing}, 65(12):3120--3135,
  2017.

\bibitem[LO14]{Lai14}
Rongjie Lai and Stanley Osher.
\newblock A splitting method for orthogonality constrained problems.
\newblock {\em Journal of Scientific Computing}, 58(2):431--449, 2014.

\bibitem[LSY16]{Lemon16}
Alex Lemon, Anthony Man-Cho So, and Yinyu Ye.
\newblock Low-rank semidefinite programming: Theory and applications.
\newblock {\em Foundations and Trends\textsuperscript{\textregistered} in
  Optimization}, 2(1-2):1--156, 2016.

\bibitem[MHA20]{Majumdar20}
Anirudha Majumdar, Georgina Hall, and Amir~Ali Ahmadi.
\newblock Recent scalability improvements for semidefinite programming with
  applications in machine learning, control, and robotics.
\newblock {\em Annual Review of Control, Robotics, and Autonomous Systems},
  3(1):331--360, 2020.

\bibitem[MMMO17]{Mei17}
Song Mei, Theodor Misiakiewicz, Andrea Montanari, and Roberto~Imbuzeiro
  Oliveira.
\newblock Solving {SDPs} for synchronization and maxcut problems via the
  {G}rothendieck inequality.
\newblock In {\em Proceedings of the 2017 Conference on Learning Theory},
  volume~65 of {\em Proceedings of Machine Learning Research}, pages
  1476--1515, 2017.

\bibitem[Pat98]{Pataki98}
Gábor Pataki.
\newblock On the rank of extreme matrices in semidefinite programs and the
  multiplicity of optimal eigenvalues.
\newblock {\em Mathematics of Operations Research}, 23(2):339--358, 1998.

\bibitem[SB03]{Burer03}
Renato D.C.~Monteiro Samuel~Burer.
\newblock A nonlinear programming algorithm for solving semidefinite programs
  via low-rank factorization.
\newblock {\em Mathematical Programming}, 95(2):329--357, 2003.

\bibitem[VA15]{Vandenberghe15}
Lieven Vandenberghe and Martin~S. Andersen.
\newblock Chordal graphs and semidefinite optimization.
\newblock {\em Foundations and Trends\textsuperscript{\textregistered} in
  Optimization}, 1(4):241–433, May 2015.

\bibitem[VB96]{Vandenberghe96}
Lieven Vandenberghe and Stephen Boyd.
\newblock Semidefinite programming.
\newblock {\em SIAM Review}, 38(1):49--95, 1996.

\bibitem[WCK17]{Wang17}
Po-Wei Wang, Wei-Cheng Chang, and J.~Zico Kolter.
\newblock The {M}ixing method: low-rank coordinate descent for semidefinite
  programming with diagonal constraints.
\newblock {\em arXiv preprint arXiv:1706.00476}, 2017.

\bibitem[Woo14]{Woodruff14}
David~P. Woodruff.
\newblock Sketching as a tool for numerical linear algebra.
\newblock {\em Foundations and Trends\textsuperscript{\textregistered} in
  Theoretical Computer Science}, 10(1–2):1–157, 2014.

\bibitem[WW20]{Waldspurger20}
Irène Waldspurger and Alden Waters.
\newblock Rank optimality for the {B}urer--{M}onteiro factorization.
\newblock {\em SIAM Journal on Optimization}, 30(3):2577--2602, 2020.

\bibitem[WYZ19]{Wang19}
Yu~Wang, Wotao Yin, and Jinshan Zeng.
\newblock Global convergence of admm in nonconvex nonsmooth optimization.
\newblock {\em Journal of Scientific Computing}, 78(1):29--63, 2019.

\bibitem[YP14]{You14}
Seungil You and Qiuyu Peng.
\newblock A non-convex alternating direction method of multipliers heuristic
  for optimal power flow.
\newblock In {\em 2014 IEEE International Conference on Smart Grid
  Communications (SmartGridComm)}, pages 788--793, 2014.

\bibitem[YTF{\etalchar{+}}21]{Yurtsever21}
Alp Yurtsever, Joel~A. Tropp, Olivier Fercoq, Madeleine Udell, and Volkan
  Cevher.
\newblock Scalable semidefinite programming.
\newblock {\em SIAM Journal on Mathematics of Data Science}, 3(1):171--200,
  2021.

\bibitem[ZFP{\etalchar{+}}20]{Zheng20}
Yang Zheng, Giovanni Fantuzzi, Antonis Papachristodoulou, Paul Goulart, and
  Andrew Wynn.
\newblock Chordal decomposition in operator-splitting methods for sparse
  semidefinite programs.
\newblock {\em Mathematical Programming}, 180(1):489--532, 2020.

\end{thebibliography}

\end{document}